\documentclass{amsart}
\usepackage{amsmath}
\usepackage{amssymb}
\usepackage{amsfonts}
\usepackage{mathrsfs}
\usepackage{geometry}
\usepackage{xcolor}

\allowdisplaybreaks[4]
\newtheorem{theorem}{Theorem}
\theoremstyle{plain}

\newtheorem{example}{Example}

\newtheorem{lemma}{Lemma}

\newtheorem{problem}{Problem}
\newtheorem{proposition}{Proposition}
\newtheorem{remark}{Remark}

\numberwithin{equation}{section}
\input{tcilatex}

\begin{document}

\title[Restricted testing]{Restricted testing for positive operators}

\author{Tuomas Hyt\"onen}
\address[T. Hyt\"onen]{Department of Mathematics and Statistics, P.O. Box 68 (Pietari Kalmin katu 5),
FI-00014 University of Helsinki, Finland}
\email{tuomas.hytonen@helsinki.fi}
\thanks{T.H. was supported by the Academy of Finland, project Nos.~307333 (Centre of Excellence in Analysis and Dynamics Research) and 314829 (Frontiers of singular integrals).}

\author{Kangwei Li}
\address[K. Li]{BCAM, Basque Center for Applied Mathematics, Mazarredo, 14. 48009 Bilbao Basque Country, Spain}
\curraddr{Center for Applied Mathematics, Tianjin University, Weijin Road 92, 300072 Tianjin, China}
\email{kangwei.nku@gmail.com}
\thanks{K.L. was supported by Juan de la Cierva - Formaci\'on 2015 FJCI-2015-24547, by the Basque Government through the BERC
2018-2021 program and by Spanish Ministry of Economy and Competitiveness
MINECO through BCAM Severo Ochoa excellence accreditation SEV-2017-0718
and through project MTM2017-82160-C2-1-P funded by (AEI/FEDER, UE) and
acronym ``HAQMEC''}

\author{Eric Sawyer}
\address[E. Sawyer]{Department of Mathematics\& Statistics, McMaster University, 1280 Main Street West, Hamilton, Ontario, Canada L8S 4K1}
\email{sawyer@mcmaster.ca}
\thanks{E.S. was supported in part by a grant from the National Sciences and Engineering Reseach Council of Canada}

\makeatletter
\@namedef{subjclassname@2010}{%
  \textup{2010} Mathematics Subject Classification}
\makeatother

\begin{abstract}
We prove that for certain positive operators $T$, such as the
Hardy-Littlewood maximal function and fractional integrals, there is a
constant $D>1$, depending only on
the dimension $n$, such that the two
weight norm inequality%
\begin{equation*}
\int_{\mathbb{R}^{n}}T\left( f\sigma \right) ^{2}d\omega \leq C\int_{\mathbb{%
R}^{n}}f^{2}d\sigma
\end{equation*}%
holds for all $f\geq 0$ \emph{if and only if} the (fractional) $A_{2}$
condition holds, and the restricted testing condition%
\begin{equation*}
\int_{Q}T\left( 1_{Q}\sigma \right) ^{2}d\omega \leq C\left\vert
Q\right\vert _{\sigma }
\end{equation*}%
holds for all cubes $Q$ satisfying $\left\vert 2Q\right\vert _{\sigma }\leq
D\left\vert Q\right\vert _{\sigma }$. If $T$ is linear, we require as well
that the dual restricted testing condition%
\begin{equation*}
\int_{Q}T^{\ast }\left( 1_{Q}\omega \right) ^{2}d\sigma \leq C\left\vert
Q\right\vert _{\omega }
\end{equation*}%
holds for all cubes $Q$ satisfying $\left\vert 2Q\right\vert _{\omega }\leq
D\left\vert Q\right\vert _{\omega }$.
\end{abstract}

\maketitle

%\tableofcontents

\section{Introduction}

One of the earliest uses of testing conditions to characterize a weighted
norm inequality occurs in 1982 in \cite{Saw3}, where it was shown that for
the Hardy-Littlewood maximal function $M$,
\begin{equation}
\int_{\mathbb{R}^{n}}Mf\left( x\right) ^{2}w\left( x\right) dx\leq C\int_{%
\mathbb{R}^{n}}f\left( x\right) ^{2}v\left( x\right) dx,\ \ \ \ \ \text{for
all }f\left( x\right) \geq 0,  \label{eq:twoweightM}
\end{equation}%
if and only if the following testing condition holds:%
\begin{equation*}
\int_{Q}M\left( \mathbf{1}_{Q}v^{-1}\right) \left( x\right) ^{2}w\left(
x\right) dx\leq C\int_{Q}v\left( x\right) ^{-1}dx,\ \ \ \ \ \text{for all
cubes }Q\text{ in }\mathbb{R}^{n}.
\end{equation*}%
Thus it suffices to test the weighted norm inequality over the simpler
collection of test functions $f=\mathbf{1}_{Q}v^{-1}$ for cubes $Q$.

Two years later, David and Journ\'{e} showed in their $T1$ theorem \cite%
{DaJo}, that the unweighted inequality%
\begin{equation*}
\int_{\mathbb{R}^{n}}Tf\left( x\right) ^{2}dx\leq C\int_{\mathbb{R}%
^{n}}f\left( x\right) ^{2}dx,\ \ \ \ \ \text{for all }f\in L^{2}\left(
\mathbb{R}^{n}\right) ,
\end{equation*}%
holds if and only if the following pair of dual testing conditions hold:%
\begin{equation*}
\int_{Q}T\left( \mathbf{1}_{Q}\right) \left( x\right) ^{2}dx\leq C\int_{Q}dx%
\text{ and }\int_{Q}T^{\ast }\left( \mathbf{1}_{Q}\right) \left( x\right)
^{2}dx\leq C\int_{Q}dx,\ \ \ \ \ \text{for all cubes }Q\text{ in }\mathbb{R}%
^{n}.
\end{equation*}%
Here $T$ is a general Calder\'{o}n-Zygmund singular integral on $\mathbb{R}%
^{n}$ and the testing functions are simply the indicators $\mathbf{1}_{Q}$
for cubes $Q$\footnote{%
The moniker `$T1$ theorem' refers to the equivalent formulation of the
testing conditions as weak boundedness property, $T1\in BMO$ and $T^{\ast
}1\in BMO$.}. The following year David, Journ\'{e} and Semmes extended the $%
T1$ theorem to a $Tb$ theorem \cite{DaJoSe} in which the testing conditions
become $b\mathbf{1}_{Q}$ and $b^{\ast }\mathbf{1}_{Q}$ for appropriately
accretive functions $b$ and $b^{\ast }$ on $\mathbb{R}^{n}$.

A couple of decades later, and motivated by the Painlev\'{e} problem of
characterizing removable singularities for bounded analytic functions,
Nazarov, Treil and Volberg solved in 2003 a particular one-weight
formulation of the norm inequality for Riesz transforms $\mathcal{R}$,
including the Cauchy transform $\mathcal{C}g\left( z\right) \equiv \int_{%
\mathbb{C}}\frac{1}{w-z}g\left( w\right) dw$ \cite{NTV},%
\begin{equation*}
\int_{\mathbb{R}^{n}}\left\vert \mathcal{R}\left( f\mu \right) \left(
x\right) \right\vert ^{2}d\mu \left( x\right) \leq C\int_{\mathbb{R}%
^{n}}f\left( x\right) ^{2}d\mu \left( x\right) ,\ \ \ \ \ \text{for all }%
f\in L^{2}\left( \mathbb{R}^{n};\mu \right) ,
\end{equation*}%
if and only if the following testing condition held:
\begin{equation*}
\int_{Q}\left\vert \mathcal{R}\left( \mathbf{1}_{Q}\mu \right) \left(
x\right) \right\vert ^{2}d\mu \left( x\right) \leq C\int_{Q}d\mu \left(
x\right) ,\ \ \ \ \ \text{for all cubes }Q\text{ in }\mathbb{R}^{n}.
\end{equation*}%
Here the testing functions are $f=\mathbf{1}_{Q}$. The Painlev\'{e} problem
was solved\ in the same year by Tolsa \cite{Tol}, a culmination of an
impressive body of work by many mathematicians.

Finally, building on the work of Nazarov, Treil and Volberg in their 2004
paper \cite{NTV3} on the Hilbert transform, that in turn used the random
dyadic grids of \cite{NTV}\footnote{%
that in turn followed on those of Fefferman and Stein \cite{FeSt}, Garnett
and Jones \cite{GaJo}, and Sawyer \cite{Saw3}}, and the weighted Haar
wavelets of \cite{NTV}\footnote{%
that in turn followed on those of Coifman, Jones and Semmes \cite{CoJoSe}},
the two weight norm inequality for the Hilbert transform was characterized
in 2014 in Lacey, Sawyer, Shen and Uriarte-Tuero \cite{LaSaShUr3}, Lacey
\cite{Lac} and Hyt\"{o}nen \cite{Hyt2} as follows:

\begin{equation}
\int_{\mathbb{R}^{n}}H\left( f\sigma \right) \left( x\right) ^{2}d\omega
\left( x\right) \leq C\int_{\mathbb{R}^{n}}f\left( x\right) ^{2}d\sigma
\left( x\right) ,\ \ \ \ \ \text{for all }f\in L^{2}\left( \sigma \right) ,
\label{eq:twoweightH}
\end{equation}%
if and only if both the strong Muckenhoupt $A_{2}$ condition%
\begin{equation*}
\mathcal{A}_{2}\left( \sigma ,\omega \right) \equiv \sup_{I:\ \text{%
intervals in $\mathbb{R}$}}\int_{\mathbb{R}}\frac{\ell \left( I\right) }{%
\left( \ell \left( I\right) +\left\vert x-c_{I}\right\vert \right) ^{2}}%
d\omega \left( x\right) \cdot \int_{\mathbb{R}}\frac{\ell \left( I\right) }{%
\left( \ell \left( I\right) +\left\vert x-c_{I}\right\vert \right) ^{2}}%
d\sigma \left( x\right) <\infty ,
\end{equation*}%
and the following dual testing conditions hold:%
\begin{equation*}
\int_{I}H\left( \mathbf{1}_{I}\sigma \right) \left( x\right) ^{2}d\omega
\left( x\right) \leq A\int_{I}d\sigma \text{ and }\int_{I}H\left( \mathbf{1}%
_{I}\omega \right) \left( x\right) ^{2}d\sigma \left( x\right) \leq
C\int_{I}d\omega ,\ \ \ \ \ \text{for all intervals }I.
\end{equation*}%
This is also referred as the two weight $T1$ theorem. Notice that here the
two weight inequality \eqref{eq:twoweightH} is written differently than %
\eqref{eq:twoweightM}, so that in the testing condition, one can avoid
requesting the existence of the density of the measure. On the other hand,
the main difference between the two weight $T1$ theorem and the unweighted $%
T1$ theorem is that, in the unweighted case, if one has $L^{2}$ boundedness
for singular integral operators, one also get $L^{p}$ boundedness for all $%
1<p<\infty $, which is not the case in the two weight setting.

The two-weight inequality for the $g$ function was then characterized by
testing conditions in Lacey and Li \cite{LaLi}, and a further extension to a
local $Tb$ theorem for the Hilbert transform is in \cite{SaShUr3}.

\begin{description}
\item[Point of departure] The point of departure for the present paper
begins with the observation that in the one-weight formulation above of the
norm inequality for Riesz transforms by Nazarov, Treil and Volberg, their
testing condition is%
\begin{equation*}
\int_{Q}\left\vert \mathcal{R}\left( \mathbf{1}_{Q}\mu \right) \left(
x\right) \right\vert ^{2}d\mu \left( x\right) \leq C\int_{2Q}d\mu \left(
x\right) ,\ \ \ \ \ \text{for all cubes }Q\text{ in }\mathbb{R}^{n},
\end{equation*}%
where the double $2Q$ of the cube $Q$ appears on the right hand side.
Moreover, one may restrict the testing functions to those functions $f=%
\mathbf{1}_{Q}$ for which $Q$ is a $\mu $\emph{-doubling cube} for some
appropriate positive constant $D$\footnote{%
This philosophy was successfully carried out in the context of the
one-weight $Tb$ theorem for nonhomogeneous square functions by Martikainen,
Mourgoglou and Vuorinen in \cite{MaMoVu}.}:%
\begin{equation*}
\int_{2Q}d\mu \leq D\int_{Q}d\mu .
\end{equation*}%
This then motivates the following problem.
\end{description}

\begin{problem}
To what extent one can similarly restrict testing functions to doubling
cubes for classical operators in the two-weight situations, including those
discussed above?
\end{problem}

%An initial step in the two-weight setting was taken by the authors in \cite%
%{LiSa}, where it was shown that such a restriction to doubling cubes is
%possible in the two weight norm inequality for fractional integrals. The
%maximal function $M$ was also considered in \cite{LiSa}, but only a much
%weaker result along these lines was obtained. The purpose of this
%paper is to prove the full result for $M$, namely that it suffices to
%restrict testing to doubling cubes in the two weight norm inequality for $M$.

\begin{description}
\item[Motivation] Besides the intrinsic interest in minimizing the functions
over which an inequality must be tested in order to verify its validity,
even a partial resolution of the question of restricted testing for singular
integrals has the potential to characterize two weight norm inequalities for
such operators - including Riesz transforms in higher dimensions, currently
a very difficult open problem, see e.g. \cite{SaShUr7}, \cite{LaWi} and \cite%
{LaSaShUrWi}. Indeed, the \emph{nondoubling cubes} have traditionally been
viewed as the enemy in two weight inequalities for singular integrals, and
(the techniques used in) the restriction of the testing conditions to just
doubling cubes could help circumvent the difficulty that \emph{energy
conditions} fail to be necessary for two weight inequalities in higher
dimensions \cite{Saw3} - the point being that a similar restricted energy
condition could suffice. In the current paper, we only work on positive
operators, such as the Hardy-Littlewood maximal function and fractional
integrals, leaving square functions and the Hilbert transform for future
work.
\end{description}

Let $\mathcal{P}^{n}$ be the collection of cubes in $\mathbb{R}^{n}$ with
sides parallel to the coordinate axes, and with side lengths $2^{\ell }$ for
some $\ell \in \mathbb{Z}$. For $Q\in \mathcal{P}^{n}$ and $\Gamma \geq 1$,
let $\Gamma Q$ denote the cube concentric with $Q$ with $\ell \left( \Gamma
Q\right) =\Gamma \ell \left( Q\right) $. For a locally signed measure $\mu $
on $\mathbb{R}^{n}$ (meaning the total variation $\left\vert \mu \right\vert
$ of $\mu $ is locally finite), we define $M(\mu )$ and $I_{\alpha }(\mu )$
at $x\in \mathbb{R}^{n}$ by\footnote{%
The supremum over $Q\in \mathcal{P}^{n}$ used here is pointwise equivalent
to the usual supremum over all cubes $Q$ with sides parallel to the
coordinate axes.}
\begin{equation*}
M(\mu )\left( x\right) \equiv \sup_{\,Q\in \mathcal{P}^{n}:\ x\in Q}\frac{1}{%
\left\vert Q\right\vert }\int_{Q}d\left\vert \mu \right\vert ,\ \ I_{\alpha
}(\mu )(x):=\int_{\mathbb{R}^{n}}\frac{1}{|x-y|^{n-\alpha }}d\mu (y).
\end{equation*}%
Given a pair $\left( \sigma ,\omega \right) $ of weights (i.e. positive
Borel measures) in $\mathbb{R}^{n}$ and $\Gamma >1$, we say that $\left(
\sigma ,\omega \right) $ satisfies the $\Gamma $\emph{-testing condition}
for the maximal function $M$ if there is a constant $\mathfrak{T}_{M}\left(
\Gamma \right) \left( \sigma ,\omega \right) $ such that
\begin{equation}
\int_{Q}\left\vert M\left( \mathbf{1}_{Q}\sigma \right) \right\vert
^{2}d\omega \leq \mathfrak{T}_{M}\left( \Gamma \right) \left( \sigma ,\omega
\right) ^{2}\left\vert \Gamma Q\right\vert _{\sigma }\ ,\ \ \ \ \ \text{for
all }Q\in \mathcal{P}^{n}\ ,  \label{Gamma testing}
\end{equation}%
and we take $\mathfrak{T}_{M}\left( \Gamma \right) \left( \sigma ,\omega
\right) $ to be the least such constant. We define $\mathfrak{T}_{I_{\alpha
}}\left( \Gamma \right) \left( \sigma ,\omega \right) $ anagolously.

There is also the following weaker testing condition, in which one need only
test the inequality over cubes that are `doubling'. Given a pair $\left(
\sigma ,\omega \right) $ of weights in $\mathbb{R}^{n}$ and $D,\Gamma >1$,
we say that $\left( \sigma ,\omega \right) $ satisfies the $D$\emph{-}$%
\Gamma $\emph{-testing condition} for the maximal function $M$ if there is a
constant $\mathfrak{T}_{M}^{D}\left( \Gamma \right) \left( \sigma ,\omega
\right) $ such that
\begin{equation}
\int_{Q}\left\vert M\mathbf{1}_{Q}\sigma \right\vert ^{2}d\omega \leq
\mathfrak{T}_{M}^{D}\left( \Gamma \right) \left( \sigma ,\omega \right)
^{2}\left\vert Q\right\vert _{\sigma }\ ,\ \ \ \ \ \text{for all }Q\in
\mathcal{P}^{n}\text{ with }\left\vert \Gamma Q\right\vert _{\sigma }\leq
D\left\vert Q\right\vert _{\sigma }\ ,  \label{D Gamma testing}
\end{equation}%
and again we take $\mathfrak{T}_{M}^{D}\left( \Gamma \right) \left( \sigma
,\omega \right) $ to be the least such constant, and define $\mathfrak{T}%
_{I_{\alpha }}^{D}\left( \Gamma \right) \left( \sigma ,\omega \right) $
similarly. Note that the $\Gamma $-testing condition implies the $D$-$\Gamma
$-testing condition for all $D>1$.

Unlike the cases of the classical two weight theorem for the maximal
function and fractional integrals in \cite{Saw3, Saw2}, where the testing
condition is already sufficient for the boundedness of the maximal function
and fractional integrals, these restricted testing conditions are not by
themselves sufficient for the norm inequality - the classical Muckenhoupt
condition is needed as well:%
\begin{equation*}
A_{2}\left( \sigma ,\omega \right) :=\sup_{Q\in \mathcal{P}^{n}}\frac{\sigma
(Q)}{\left\vert Q\right\vert }\frac{\omega (Q)}{\left\vert Q\right\vert }%
<\infty ,\ \ A_{2}^{\alpha }\left( \sigma ,\omega \right) :=\sup_{Q\in
\mathcal{P}^{n}}\frac{\sigma (Q)}{\left\vert Q\right\vert ^{1-\frac{\alpha }{%
n}}}\frac{\omega (Q)}{\left\vert Q\right\vert ^{1-\frac{\alpha }{n}}}<\infty
\end{equation*}%
see the counterexample in Section \ref{sec: counterexample maximal} and
Section \ref{sec: frac}, respectively. Finally we let $\mathfrak{N}%
_{M}\left( \sigma ,\omega \right) $ be the best constant (i.e., the $%
L^{2}\left( \sigma \right) \rightarrow L^{2}\left( \omega \right) $ norm of $%
M$) in the inequality%
\begin{equation*}
\int_{\mathbb{R}^{n}}\left\vert M\left( f\sigma \right) \right\vert
^{2}d\omega \leq \mathfrak{N}_{M}\left( \sigma ,\omega \right) ^{2}\int_{%
\mathbb{R}^{n}}\left\vert f\right\vert ^{2}d\sigma ,\ \ \ \ \ \text{for all }%
f\in L^{2}\left( \sigma \right) .
\end{equation*}%
Again we define $\mathfrak{N}_{I_{\alpha }}\left( \sigma ,\omega \right) $
analogously. Our main result for the maximal function is formulated as the
following.

\begin{theorem}
\label{weak}Let $\Gamma >1$. Then there is $D>1$ depending only on $\Gamma $
and the dimension $n$ such that%
\begin{equation*}
\mathfrak{N}_{M}\left( \sigma ,\omega \right) \approx \mathfrak{T}%
_{M}^{D}\left( \Gamma \right) \left( \sigma ,\omega \right) +\sqrt{%
A_{2}\left( \sigma ,\omega \right) },
\end{equation*}%
for all locally finite positive Borel measures $\sigma $ and $\omega $ on $%
\mathbb{R}^{n}$.
\end{theorem}

\begin{remark}
With probability one, a dyadic grid $\mathcal{D}^{\beta }$ in (\ref{def
infinite dyadic grid}) below has the property that every $Q\in \mathcal{D}%
^{\beta }$ has \emph{null boundary}, i.e. $\left\vert \partial Q\right\vert
_{\sigma +\omega }=0$. Thus the supremum over cubes $Q$ in the testing
constant $\mathfrak{T}_{M}^{D}\left( \Gamma \right) \left( \sigma ,\omega
\right) $ in Theorem \ref{weak} may be further restricted to cubes $Q$
having null boundary (\emph{cf.} the one-weight theorem in \cite{MaMoVu}
where this type of reduction first appears).
\end{remark}

See also \cite{LiSa}, \cite{LiSa2} and \cite{ChLa} for earlier related work.
For example, the following weaker `parental' testing result for the maximal
function was proved in \cite{LiSa}, and subsequently given a particularly
simple proof by Chen and Lacey in \cite{ChLa}: The two weight norm
inequality for the maximal function $M$ holds if and only if the following $%
D $\emph{-parental testing condition} holds for some $D>1$,%
\begin{equation*}
\int_{Q}\left\vert M\mathbf{1}_{Q}\sigma \right\vert ^{2}d\omega \leq
\mathfrak{P}_{T}^{D}\left( \sigma ,\omega \right) ^{2}\left\vert
Q\right\vert _{\sigma },\text{ \ for all }Q\in \mathcal{P}^{n}\text{ with }%
\min_{P\text{ is a dyadic parent of }Q}\left\vert P\right\vert _{\sigma
}\leq D\left\vert Q\right\vert _{\sigma }\ .
\end{equation*}
Let us say that a cube $Q$ is {\em $D$-$\Gamma$-doubling} if it satisfies $|\Gamma Q|_\sigma\leq D|Q|_{\sigma}$ (as in \eqref{D Gamma testing}) and the $Q$ is
 {\em $D$-parental doubling} if it satisfies the condition at the end of the previous display. It is easy to see that every $D$-$3$-doubling cube is $D$-parental doubling, but the converse is false in general: $D$-parental doubling means that the measure $\sigma$ is under control on {\em at least one side of $Q$}, while $D$-$\Gamma$-doubling means that it is under control on {\em all sides}. Thus there are many more cubes that one needs to test in the $D$-parental testing condition than in the $D$-$\Gamma$-testing condition \eqref{D Gamma testing}.

The proof of Theorem \ref{weak} splits neatly into two steps. In the first
step, we prove the sufficiency of the $\Gamma $-testing condition
\eqref{Gamma
testing}, which we record as the following theorem

\begin{theorem}
\label{maximal}For $\Gamma >1$ we have%
\begin{equation*}
\mathfrak{N}_{M}\left( \sigma ,\omega \right) \approx \mathfrak{T}_{M}\left(
\Gamma \right) \left( \sigma ,\omega \right) +\sqrt{A_{2}\left( \sigma
,\omega \right) },
\end{equation*}%
for all pairs $\left( \sigma ,\omega \right) $ of locally finite positive
Borel measures on $\mathbb{R}^{n}$, and where the implicit constants of
comparability depend on both $\Gamma $ and dimension $n$.
\end{theorem}

This step requires a careful application of a probabilistic argument of the
type pioneered by Nazarov, Treil and Volberg (\cite{NTV}), and refined in
\cite{Hyt2}. In the second step we use this interim result to establish an
\textit{a priori} bound on the operator norm $\mathfrak{N}_{M}\left( \sigma
,\omega \right) $ in order to absorb additional terms arising from the
absence of any testing condition at all in (\ref{D Gamma testing}) when the
cubes are not doubling. As a consequence of this splitting, we will give the
proof in two stages, beginning with the proof of the following weaker
theorem, which requires probability, and which is then used to prove our
main result Theorem \ref{weak}. We emphasize that this paper is
self-contained.

Our main result for fractional integrals is the following theorem.

\begin{theorem}
\label{weakfrac}Let $\Gamma >1$. Then there is $D>1$ depending only on $%
\Gamma $ and the dimension $n$ such that%
\begin{equation*}
\mathfrak{N}_{I_\alpha}\left( \sigma ,\omega \right) \approx \mathfrak{T}%
_{I_\alpha}^{D}\left( \Gamma \right) \left( \sigma ,\omega \right) +%
\mathfrak{T}_{I_\alpha}^{D}\left( \Gamma \right) \left( \omega,\sigma
\right) +\sqrt{A_{2}^\alpha\left( \sigma ,\omega \right) },
\end{equation*}%
for all locally finite positive Borel measures $\sigma $ and $\omega $ on $%
\mathbb{R}^{n}$.
\end{theorem}

For convenience we will restrict our proof of the above results to the case $%
\Gamma =3$, the general case of $\Gamma $ large being an easy modification
of this one.

\section{Preliminaries}

Here we introduce some standard tools we will use in the proof of Theorem %
\ref{maximal}.

\subsection{Random dyadic grids\label{Sub dyadic}}

In this subsection we introduce the usual random dyadic grids. Let
\begin{equation*}
\mathcal{D}_{0}:=\{2^{-j}([0,1)^{n}+k),j\in \mathbb{Z},k\in \mathbb{Z}^{n}\}.
\end{equation*}%
Then for $\beta =\{\beta _{j}\}_{j=-\infty }^{\infty }\in (\{0,1\}^{n})^{%
\mathbb{Z}}$, define
\begin{equation}
\mathcal{D}^{\beta }:=\left\{ Q+\sum_{j:2^{-j}<\ell (Q)}2^{-j}\beta
_{j},Q\in \mathcal{D}_{0}\right\} .  \label{def infinite dyadic grid}
\end{equation}%
Let $\mathbb{P}$ be the natural product probability measure on $\Omega
:=(\{0,1\}^{n})^{\mathbb{Z}}$. We have the following estimate, which goes
back to Fefferman and Stein \cite[page 112]{FeSt} and also \cite[Lemma 2]%
{Saw3}. For any $\beta\in \Omega $, we denote the associated \emph{dyadic}
maximal operator by
\begin{equation*}
M^{\mathcal{D}^{\beta}}f(x):= \sup_{Q\ni x, Q\in \mathcal{D}^\beta}\frac{1}{\left\vert
Q\right\vert }\int_{Q}\left\vert f\right\vert .
\end{equation*}

\begin{lemma}
\label{domination}For $x\in \mathbb{R}^{n}$ and a positive Borel measure $%
f\geq 0$ on $\mathbb{R}^{n}$ we have%
\begin{equation}
Mf\left( x\right) \leq 2^{4n+1}\mathbb{E}_\beta M^{\mathcal{D}^\beta}f\left(
x\right) .  \label{lim inf average}
\end{equation}
\end{lemma}

\begin{proof}
Fix $x\in \mathbb{R}^{n}$, and let $Q$ be a cube such that $x\in Q$ and
\begin{equation*}
\frac{1}{\left\vert Q\right\vert }\int_{Q}f>\frac{1}{2}Mf\left( x\right) .
\end{equation*}%
Then there exists $Q\subset \widetilde Q \in \mathcal{D}^{\gamma}$ for some $%
\gamma\in \Omega$ such that
\begin{equation*}
\frac{1}{\vert \widetilde Q\vert }\int_{\widetilde Q}f>\frac{1}{2^{n+1}}%
Mf\left( x\right).
\end{equation*}
Since $\gamma$ is fixed, denote $j_0= -\log_2 \ell(\widetilde Q)+1$, we have
\begin{equation*}
\mathbb{P }(\{\beta\in \Omega: \gamma_{j_0}+ \beta_{j_0}=(1, \cdots,
1)\})=2^{-n}.
\end{equation*}
The key is when $\gamma_{j_0}+ \beta_{j_0}=(1, \cdots, 1)$, suppose
\begin{equation*}
\widetilde Q-\sum_{j: 2^{-j}<\ell(\widetilde Q)}2^{-j} \gamma_j+\sum_{j:
2^{-j}<\ell(\widetilde Q)}2^{-j}\beta_j=[a_1, b_1)\times \cdots \times [a_n,
b_n),
\end{equation*}
let $\widetilde Q_{\beta}=J_1\times \cdots \times J_n$ be the cube
satisfying that $J_i= [a_i, 2b_i-a_i)$ if $(\gamma_{j_0})_i=1$ and $%
J_i=[2a_i-b_i, b_i)$ if $(\gamma_{j_0})_i=0$, then $\widetilde
Q_{\beta}\supset \widetilde Q$ and $\widetilde Q_{\beta}\in \mathcal{D}%
^\beta $ with some fixed $\beta_{j_0-1}$ (whose precise value depends on $%
\gamma_{j_0}$ and the standard dyadic cube $\widetilde Q-\sum_{j:
2^{-j}<\ell(\widetilde Q)}2^{-j} \gamma_j$). This implies
\begin{equation*}
\mathbb{P}(\{\beta\in \Omega: M^{\mathcal{D}^\beta}f> \frac 1{2^{2n+1}}
Mf\})\ge 4^{-n}
\end{equation*}
which completes the proof of (\ref{lim inf average}).
\end{proof}

\subsection{Whitney decompositions}

In order to apply probabilistic method, here we use a slightly weaker
version of Whitney decomposition, which adapts to the probabilistic method
quite naturally. Indeed, given an open set $\Omega$ we let $\{Q_j\}_j$ be
the collection of dyadic cubes such that

\begin{enumerate}
\item $Q_j\subset \Omega$;

\item $10 \sqrt n\ell(Q_j)<\mathop{\rm{dist}}(Q_j, \Omega^c)\le 21\sqrt
n\ell(Q_j)$.
\end{enumerate}

Notice that we do not request $Q_j$ to be the maximal dyadic cube such that
the above properties holds. With this definition, we prove the following
proposition.

\begin{proposition}
\label{prop:whitney} Let $\Omega$ be an open set and $\{Q_j\}$ be the
related Whitney cubes defined as the above. Then there holds

\begin{enumerate}
\item[(i)] $\Omega\subset \cup_j Q_j$;

\item[(ii)] $\sum_j \chi_{Q_j}\le 2$;

\item[(iii)] $\sum_j \chi_{3Q_j}\le c_n$.
\end{enumerate}
\end{proposition}

\begin{proof}
To prove the first assertion, given $x$ and a dyadic cube $Q$ with $x\in
Q\subsetneq \Omega$, notice that
\begin{align*}
\frac 12\cdot\frac{\mathop{\rm{dist}}(Q, \Omega^c)}{\ell(Q)} -\frac{\sqrt n}%
2 \le \frac{\mathop{\rm{dist}}(Q^{(1)}, \Omega^c)}{2\ell(Q)} \le \frac
12\cdot\frac{\mathop{\rm{dist}}(Q, \Omega^c)}{\ell(Q)}.
\end{align*}
Obviously, we can take a dyadic cube $Q_0\ni x$ such that $\frac{%
\mathop{\rm{dist}}(Q_0, \Omega^c)}{\ell(Q_0)}$ is sufficiently big, with the
above estimates, there must exists a cube $Q_x\ni x$ such that
\begin{equation*}
10 \sqrt n\ell(Q_x)<\mathop{\rm{dist}}(Q_x, \Omega^c)\le 21\sqrt n\ell(Q_x).
\end{equation*}
Now we turn to prove (ii). Fix $x\in \Omega$, let $d=\mathop{\rm{dist}}(x,
\Omega^c) $ and $x\in Q_j$, then
\begin{equation*}
d\ge \mathop{\rm{dist}}(Q_j, \Omega^c) > 10 \sqrt n\ell(Q_j).
\end{equation*}
On the other hand,
\begin{align*}
d\le \mathop{\rm{dist}}(Q_j, \Omega^c) + \sqrt n \ell(Q_j)\le 22\sqrt n
\ell(Q_j).
\end{align*}
Combing the above two estimates one obtains that
\begin{equation*}
\frac d{22\sqrt n }\le \ell(Q_j)< \frac d{10\sqrt n}.
\end{equation*}
Now since $Q_j$ is a dyadic cube, we get (ii). Finally, let us fix $x$ and $%
3Q_j\ni x$. If $Q_l$ is a Whitney cube such that $x\in 3Q_l$, then we have
\begin{equation*}
21\sqrt n\ell(Q_l)\ge \mathop{\rm{dist}}(Q_l, \Omega^c)\ge \mathop{\rm{dist}}%
(Q_j, \Omega^c)- 2 \sqrt n (\ell(Q_j)+ \ell(Q_l))> 8\sqrt n \ell(Q_j)-2\sqrt
n \ell(Q_l).
\end{equation*}%
Likewise, we also have $23\ell(Q_j)>8\ell(Q_l)$. Since $\mathop{\rm{dist}}%
(Q_j, Q_l)\le 2\sqrt n(\ell(Q_j)+\ell(Q_l))$, we completes the proof of
(iii).
\end{proof}

\subsection{Good/bad cubes}

We call a cube $Q\in \mathcal{D}^{\gamma }$ is good, if for any $R\in
\mathcal{D}^{\gamma }$ with $\ell (R)\geq 2^{r}\ell (Q)$, there holds
\begin{equation*}
\mathrm{dist}(Q,\partial R)>\ell (Q)^{\frac{1}{2}}\ell (R)^{\frac{1}{2}}.
\end{equation*}%
Otherwise we call $Q$ is bad. We note that $\pi _{\mathrm{good}}:=\mathbb{P}%
_{\gamma }(Q+\sum_{j:2^{-j}<\ell (Q)}2^{-j}\gamma _{j}$ is good$)$ is
independent of $Q\in \mathcal{D}_{0}$. It is well-known that if $r$ is large
enough, then $\pi _{\mathrm{good}}>0$. So without loss of generality, we can
assume $r>4$.

\section{Strong triple testing}

\label{sec: proof maximal}

\subsection{A counterexample}

\label{sec: counterexample maximal} We begin this section with the following
counterexample, which shows that it is necessary to assume the Muckenhoupt $%
A_2$ condition for the characterization.

\begin{example}
\label{no control}Define%
\begin{eqnarray*}
d\sigma \left( y\right) &=&e^{y}dy, \\
d\omega \left( x\right) &=&\mathbf{1}_{\left[ 0,1\right] }\left( x\right) dx.
\end{eqnarray*}%
Then%
\begin{equation*}
\mathfrak{N}_{M}\left( \sigma ,\omega \right) \geq A_{2}\left( \sigma
,\omega \right) \geq \sup_{R>1}\frac{\left\vert \left[ 0,R\right]
\right\vert _{\omega }}{\left\vert R\right\vert }\frac{\left\vert \left[ 0,R%
\right] \right\vert _{\sigma }}{\left\vert R\right\vert }=\sup_{R>1}\frac{1}{%
R}\frac{e^{R}-1}{R}=\infty ,
\end{equation*}%
and%
\begin{equation*}
\mathfrak{T}_{M}\left( 3\right) \left( \sigma ,\omega \right) \lesssim 1.
\end{equation*}%
Indeed, without loss of generality, $I=\left[ a,b\right] $ with $I\cap \left[
0,1\right] \neq \emptyset $ (since otherwise $\mathbf{1}_{I}\omega =0$ and $%
\int_{I}\left\vert M\left( \mathbf{1}_{I}\sigma \right) \right\vert
^{2}d\omega =0$) and so%
\begin{equation}
a<1\text{ and }b>0.  \label{necc}
\end{equation}%
Now we assume (\ref{necc}) and compute $\frac{1}{\left\vert 3I\right\vert
_{\sigma }}\int_{I}\left\vert M\left( \mathbf{1}_{I}\sigma \right)
\right\vert ^{2}d\omega $ in two cases.

\begin{enumerate}
\item Case $b>2$: In this case we have $M\left( \mathbf{1}_{I}\sigma \right)
\left( x\right) =\frac{\int_{x}^{b}e^{y}dy}{b-x}\leq \frac{e^{b}-1}{b-1}$
for $0\leq x\leq 1$, and so
\begin{eqnarray*}
&&\int_{I}\left\vert M\left( \mathbf{1}_{I}\sigma \right) \right\vert
^{2}d\omega \leq \int_{0}^{1}\left( \frac{e^{b}-1}{b-1}\right) ^{2}dx\approx
\frac{e^{2b}}{b^{2}}, \\
&&\left\vert 3I\right\vert _{\sigma }=\int_{2a-b}^{2b-a}d\sigma \geq
\int_{2b-a-1}^{2b-a}e^{y}dy\approx e^{2b-a}, \\
&\Longrightarrow &\frac{1}{\left\vert 3I\right\vert _{\sigma }}%
\int_{I}\left\vert M\left( \mathbf{1}_{I}\sigma \right) \right\vert
^{2}d\omega \lesssim \frac{\frac{e^{2b}}{b^{2}}}{e^{2b-a}}=\frac{e^{a}}{b^{2}%
}\leq 1.
\end{eqnarray*}

\item Case $b\leq 2$: In this case we have $M\left( \mathbf{1}_{I}\sigma
\right) \left( x\right) =\frac{\int_{x}^{b}e^{y}dy}{b-x}\leq e^{2}$ for $%
0\leq x\leq 1$, and so we consider two subcases.

\begin{enumerate}
\item Subcase $a\geq -1$:%
\begin{eqnarray*}
&&\int_{I}\left\vert M\left( \mathbf{1}_{I}\sigma \right) \right\vert
^{2}d\omega \leq e^{2}\left\vert I\cap \left[ 0,1\right] \right\vert \\
&&\left\vert 3I\right\vert _{\sigma }=\int_{2a-b}^{2b-a}e^{y}dy\geq
e^{2a-b}3\left( b-a\right) \geq 3e^{-4}\left( b-a\right) , \\
&\Longrightarrow &\frac{1}{\left\vert 3I\right\vert _{\sigma }}%
\int_{I}\left\vert M\left( \mathbf{1}_{I}\sigma \right) \right\vert
^{2}d\omega \leq \frac{e^{2}\left\vert I\cap \left[ 0,1\right] \right\vert }{%
3e^{-4}\left( b-a\right) }\leq \frac{e^{6}}{3}.
\end{eqnarray*}

\item Subcase $a<-1$: In this subcase we also have $b-a>0-\left( -1\right)
=1 $ and so%
\begin{eqnarray*}
&&\int_{I}\left\vert M\left( \mathbf{1}_{I}\sigma \right) \right\vert
^{2}d\omega \leq e^{2}\left\vert I\cap \left[ 0,1\right] \right\vert \\
&&\left\vert 3I\right\vert _{\sigma }=\int_{2a-b}^{2b-a}e^{y}dy\geq
\int_{b-1}^{b}e^{y}dy=e^{b}-e^{b-1}\geq 1-e^{-1}, \\
&\Longrightarrow &\frac{1}{\left\vert 3I\right\vert _{\sigma }}%
\int_{I}\left\vert M\left( \mathbf{1}_{I}\sigma \right) \right\vert
^{2}d\omega \leq \frac{e^{2}\left\vert I\cap \left[ 0,1\right] \right\vert }{%
1-e^{-1}}\leq \frac{e^{2}}{1-e^{-1}}.
\end{eqnarray*}
\end{enumerate}
\end{enumerate}
\end{example}

On the other hand, if we interchange these measures, then we have $\mathfrak{%
T}_{M}\left( 3\right)(\omega,\sigma) =\infty $ since with $I=\left[ 0,R%
\right] $ and $R>2$, we have%
\begin{equation*}
\frac{1}{\left\vert 3I\right\vert _{\omega }}\int_{I}\left\vert M\left(
\mathbf{1}_{I}\omega \right) \right\vert ^{2}d\sigma \gtrsim \int_{\left[ 1,R%
\right] }\left\vert \frac{\int_{0}^{1}dx}{y}\right\vert
^{2}e^{y}dy=\int_{1}^{R}\frac{e^{y}}{y^{2}}dy\approx \frac{e^{R}}{R^{2}}\ .
\end{equation*}

%%%%%%%%%%%%

\subsection{Proof of Theorem \protect\ref{maximal}}

Now we begin the proof of Theorem \ref{maximal}, the `only if' part is
trivial, so we only focus on the `if' part. We shall prove
\begin{equation}
\mathfrak{N}_{M}\left( \sigma ,\omega \right) \lesssim \mathfrak{T}%
_{M}\left( 3\right) \left( \sigma ,\omega \right) +\sqrt{ A_{2}\left( \sigma
,\omega \right) }.  \label{it suffices}
\end{equation}%
Fix $f$ nonnegative and bounded with compact support, say $\mathop{\rm{supp}}%
f\subset Q(0,R)=\left[ -R,R\right] ^{n}$. Since $M\left( f\sigma \right) $
is lower semicontinuous, the set $\Omega _{k}:= \left\{ M\left( f\sigma
\right) >2^{k}\right\} $ is open and we can consider the Whitney
decomposition of the open set $\Omega _{k}$ into the union $%
\bigcup\limits_{j\in \mathbb{N}}Q_{j}^{k}$ of $\mathcal{D}^{\gamma }$-dyadic
intervals $Q_{j}^{k}$ with the properties as in Proposition \ref%
{prop:whitney}, where $\gamma \in \Omega$. In the sequel, we will use $%
\mathcal{W}_k^\gamma$ to denote the Whitney cubes of $\Omega_k$ in $\mathcal{%
D}^\gamma$. We now use random grids to obtain from Lemma \ref{domination}
that%
\begin{equation*}
M\left( f\sigma \right) \left( x\right) \lesssim \mathbb{E}_{\gamma }M^{%
\mathcal{D}^\gamma}f\left( x\right) \ ,\ \ \ \ \ x\in \mathbb{R}^{n}.
\end{equation*}

Notice that if we replace $\omega $ by $\omega _{N}=\omega \mathbf{1}%
_{Q(0,N)}$ with $N>R$, we have
\begin{equation*}
\int M\left( f\sigma \right) ^{2}d\omega _{N}\leq \Vert f\Vert _{L^{\infty
}}^{2}\int_{Q(0,N)}M(\mathbf{1}_{Q(0,N)}\sigma )^{2}d\omega \leq \Vert
f\Vert _{L^{\infty }}^{2}\mathfrak{T}_{M}^{2}\left\vert 3Q(0,N)\right\vert
_{\sigma }<\infty ,
\end{equation*}%
and therefore, without loss of generality, we can assume
\begin{equation*}
\int M(f\sigma )^{2}d\omega <\infty .
\end{equation*}
We now have
\begin{align*}
\mathbb{E}_{\gamma }\int_{\mathbb{R}^{n}}\left[ M^{\mathcal{D}^{\gamma
}}\left( f\sigma \right) \left( x\right) \right] ^{2}d\omega \left( x\right)
& \leq \mathbb{E}_{\gamma }C_{n}\sum_{k\in \mathbb{Z}}2^{2(k+m)}\left\vert
\left\{ M^{\mathcal{D}^{\gamma }}\left( f\sigma \right) >2^{k+m}\right\}
\right\vert _{\omega } \\
& =\mathbb{E}_{\gamma }C_{n}\sum_{k\in \mathbb{Z},\ j\in \mathbb{N}%
}2^{2(k+m)}\left\vert Q_{j}^{k}\cap \Omega _{k+m}^{\gamma }\right\vert
_{\omega } \\
& \leq C_{n,m}\mathbb{E}_{\gamma }\sum_{k\in \mathbb{Z},\ j\in \mathbb{N}%
}2^{2k}\left\vert E_{j,\gamma }^{k}\right\vert _{\omega
}+3^{n}C_{n}2^{-2m_{0}}\int \left[ M\left( f\sigma \right) \right]
^{2}d\omega \ ,
\end{align*}%
where%
\begin{equation*}
E_{j,\gamma }^{k}:=Q_{j}^{k}\cap \left( \Omega _{k+m}^{\gamma }\setminus
\Omega _{k+m+m_{0}}\right) ,\,\,\Omega _{k+m}^{\gamma }=\left\{ x:M^{%
\mathcal{D}^{\gamma }}\left( f\sigma \right) >2^{k+m}\right\} ,
\end{equation*}%
and we shall choose $m_{0}$ to be sufficiently large so that the second term
can be absorbed (since it is finite). So the goal is to prove
\begin{equation*}
\mathbb{E}_{\gamma }\sum_{k\in \mathbb{Z},\ j\in \mathbb{N}}2^{2k}\left\vert
E_{j,\gamma }^{k}\right\vert _{\omega }\lesssim \big( \mathfrak{T}_{M}\left(
3\right) \left( \sigma ,\omega \right) ^{2}+A_{2}\left( \sigma ,\omega
\right) \big) \Vert f\Vert _{L^{2}(\sigma )}^{2}.
\end{equation*}

We claim the maximum principle,%
\begin{equation}
2^{k+m-1}<M_{Q_{j}^{k}}^{\mathcal{D}^{\gamma }}(f\sigma )(x):=\sup_{Q\in
\mathcal{D}^{\gamma },x\in Q\subseteq Q_{j}^{k}}\frac{1}{|Q|}\int_{Q}f\sigma
,\qquad x\in E_{j,\gamma }^{k}.  \label{max princ}
\end{equation}%
Indeed,
\begin{equation}
\begin{split}
M^{\mathcal{D}^{\gamma }}(f\sigma )(x)& \leq M^{\mathcal{D}^{\gamma
}}(1_{(Q_{j}^{k})^{c}}f\sigma )(x)+M^{\mathcal{D}^{\gamma
}}(1_{Q_{j}^{k}}f\sigma )(x) \\
& \leq M^{\mathcal{D}^{\gamma }}(1_{(Q_{j}^{k})^{c}}f\sigma )(x)+\sup_{Q\in
\mathcal{D}^{\gamma }:Q\supsetneq Q_{j}^{k}}\frac{1}{|Q|}\int_{Q}f\sigma
+M_{Q_{j}^{k}}^{\mathcal{D}^{\gamma }}(f\sigma )(x).
\end{split}
\label{eq:Msplit}
\end{equation}

Given $x\in E_{j,\gamma }^{k}$, there is $Q\in \mathcal{D}^{\gamma }$ with $%
x\in Q$ and $Q\cap \left( Q_{j}^{k}\right) ^{c}\neq \emptyset $ (which
implies that $Q_{j}^{k}\subset Q$), and also $z\in \Omega _{k}^{c}$, such
that
\begin{eqnarray}
M^{\mathcal{D}^{\gamma }}\big(\mathbf{1}_{\left( Q_{j}^{k}\right)
^{c}}f\sigma \big)\left( x\right) &\leq &2\frac{1}{\left\vert Q\right\vert }%
\int_{Q\setminus Q_{j}^{k}}f\sigma \leq 2\frac{1}{\left\vert Q\right\vert }%
\int_{50\sqrt{n}Q}f\sigma  \label{z exists} \\
&=&\frac{2(50\sqrt{n})^{n}}{\left\vert 50\sqrt{n}Q\right\vert }\int_{50\sqrt{%
n}Q}f\sigma \leq 2(50\sqrt{n})^{n}M\left( f\sigma \right) \left( z\right)
\leq 2^{k+m-2}  \notag
\end{eqnarray}%
if we choose $m>1$ large enough.

This same computation shows that the cubes $Q\supsetneq Q^k_j$ in the second
term on the right of \eqref{eq:Msplit} satisfy
\begin{equation*}
\frac{1}{|Q|}\int_Q f\sigma \leq 2^{k+m-2}.
\end{equation*}

Now we use $2^{k+m}<M^{\mathcal{D}^{\gamma }}\left( f\sigma \right) \left(
x\right) $ for $x\in E_{j}^{k}$ to obtain
\begin{equation}
\begin{split}
2^{k+m}<M^{\mathcal{D}^{\gamma }}(f\sigma )(x)& \leq M^{\mathcal{D}^{\gamma
}}(1_{(Q_{j}^{k})^{c}}f\sigma )(x)+\sup_{Q\in \mathcal{D}^{\gamma
}:Q\supsetneq Q_{j}^{k}}\frac{1}{|Q|}\int_{Q}f\sigma +M_{Q_{j}^{k}}^{%
\mathcal{D}^{\gamma }}(f\sigma )(x) \\
& \leq 2^{k+m-2}+2^{k+m-2}+M_{Q_{j}^{k}}^{\mathcal{D}^{\gamma }}(f\sigma
)(x)=2^{k+m-1}+M_{Q_{j}^{k}}^{\mathcal{D}^{\gamma }}(f\sigma )(x).
\end{split}
\label{obtain}
\end{equation}

Hence we have proved the maximum principle. Thus set
\begin{equation*}
\widetilde{\Omega }_{k+m}^{\gamma }:=\{M_{Q_{j}^{k}}^{D^{\gamma }}(f\sigma
)>2^{k+m-1}\},\ \ \widetilde{E}_{j,\gamma }^{k}:=Q_{j}^{k}\cap (\widetilde{%
\Omega }_{k+m}^{\gamma }\setminus \Omega _{k+m+m_{0}})
\end{equation*}%
then $E_{j,\gamma }^{k}\subset \widetilde{E}_{j,\gamma }^{k}$ and the latter
depends only on the cube $Q_{j}^{k}$ (not on its parents or ancestors in the
dyadic system $\mathcal{D}^{\gamma }$) so in particular it is independent of
the goodness of $Q_{j}^{k}$. Thus (see e.g. \cite{HyPeTrVo})
\begin{equation*}
\mathbb{E}_{\gamma }\sum_{k,j}2^{2k}|E_{j,\gamma }^{k}|_{\omega }\leq
\mathbb{E}_{\gamma }\sum_{k,j}2^{2k}|\widetilde{E}_{j,\gamma }^{k}|_{\omega
}=\pi _{\mathrm{good}}^{-1}\mathbb{E}_{\gamma }\sum_{k,j}2^{2k}|\widetilde{E}%
_{j,\gamma }^{k}|_{\omega }1_{Q_{j}^{k}\text{ good}}.
\end{equation*}

%
%\begin{equation*}
%2^{k+m-1}<M^{\mathcal{D}^{\gamma }}\left( f\sigma \right) \left(
%x\right) -M^{\mathcal{D}^{\gamma }}\left( \mathbf{1}_{\left(
%Q_{j}^{k}\right) ^{c}}f\sigma \right) \left( x\right) \leq M^{%
%\mathcal{D}^{\gamma }}\left( \mathbf{1}_{Q_{j}^{k}}f\sigma \right) \left(
%x\right) .
%\end{equation*}

We now introduce some further notation which will play a crucial role below.
Let%
\begin{eqnarray*}
\mathcal{H}_{j}^{k}:= &&\left\{ M_{Q_{j}^{k}}^{\mathcal{D}^{\gamma }}\left(
\mathbf{1}_{Q_{j}^{k}}f\sigma \right) >2^{k+m-1}\right\} , \\
\mathcal{H}_{j,\mathrm{out}}^{k}:= &&\left\{ M_{Q_{j}^{k}}^{\mathcal{D}%
^{\gamma }}\left( \mathbf{1}_{Q_{j}^{k}\setminus \Omega _{k+m+m_{0}}}f\sigma
\right) >2^{k+m-2}\right\} , \\
\mathcal{H}_{j,\mathrm{in}}^{k}:= &&\left\{ M_{Q_{j}^{k}}^{\mathcal{D}%
^{\gamma }}\left( \mathbf{1}_{Q_{j}^{k}\cap \Omega _{k+m+m_{0}}}f\sigma
\right) >2^{k+m-2}\right\} ,
\end{eqnarray*}%
so that $\mathcal{H}_{j}^{k}\subset \mathcal{H}_{j,\mathrm{out}}^{k}\cup
\mathcal{H}_{j,\mathrm{in}}^{k}$. We are here suppressing the dependence of $%
\mathcal{H}_{j}^{k}$ on $\gamma \in \Omega $.

We will now follow the main idea for fractional integrals in \cite{Saw2},
but with two main changes:

\begin{enumerate}
\item \textbf{Sublinearizations}: Since $M$ is not linear, the duality
arguments in \cite{Saw2} require that we construct symmetric linearizations $%
L$ that are dominated by $M$, and

\item \textbf{Tripling decompositions}: In order to exploit the triple
testing conditions we introduce Whitney grids, and construct stopping times
for tripling cubes, which entails some combinatorics. In particular, most of
our effort is spent on decomposing and controlling the analogue of term $IV$
from \cite{Saw2} using good and bad cubes.
\end{enumerate}

Now take $0<\beta <1$ to be chosen later, and consider the following three
exhaustive cases for $Q_{j}^{k}$ and $\widetilde E_{j,\gamma}^{k}$.

(\textbf{1}): $Q_j^k$ is good and $\vert \widetilde E_{j,\gamma}^{k}\vert
_{\omega }<\beta \vert 3Q_{j}^{k}\vert _{\omega }$, in which case we say $%
(k,j)\in \Pi _{1}$,

(\textbf{2}): $Q_j^k$ is good and $\vert \widetilde E_{j,\gamma}^{k}\vert
_{\omega }\geq \beta \vert 3Q_{j}^{k}\vert _{\omega }$\ and $\vert
\widetilde E_{j,\gamma}^{k}\cap \mathcal{H}_{j,\mathrm{out}}^{k}\vert
_{\omega }\geq \frac{1}{2}\vert \widetilde E_{j,\gamma}^{k}\vert _{\omega }$%
, say $(k,j)\in \Pi _{2}$,

(\textbf{3}): $Q_j^k$ is good and $\vert \widetilde E_{j,\gamma}^{k}\vert
_{\omega }\geq \beta \vert 3Q_{j}^{k}\vert _{\omega }$\ and $\vert
\widetilde E_{j,\gamma}^{k}\cap \mathcal{H}_{j,\mathrm{in}}^{k}\vert
_{\omega }\geq \frac{1}{2}\vert \widetilde E_{j,\gamma}^{k}\vert _{\omega }$%
, say $(k,j)\in \Pi _{3}$.

\subsection{The three cases}

The first case is trivially handled, the second case is easy, and the third
case consumes most of our effort.

\textbf{Case (1)}: The treatment of case (1) is easy by absorption. Indeed,
\begin{equation}
\sum_{(k,j)\in \Pi _{1}}2^{2k}\vert \widetilde E_{j,\gamma}^{k}\vert
_{\omega } \lesssim \sum_{k\in \mathbb{Z},\ j\in \mathbb{N}}2^{2k}\beta
\left\vert 3Q_{j}^{k}\right\vert _{\omega }\lesssim \beta \int M\left(
f\sigma \right) ^{2}d\omega ,  \label{case 1 est}
\end{equation}%
and then it suffices to take $\beta $ sufficiently small at the end of the
proof.

\textbf{Case (2)}: In case (2) we have%
\begin{equation}
\sum_{(k,j)\in \Pi _{2}}2^{2k}|\widetilde{E}_{j,\gamma }^{k}|_{\omega
}\lesssim \sum_{(k,j)\in \Pi _{2}}2^{k}\int \mathbf{1}_{\widetilde{E}%
_{j,\gamma }^{k}}\mathcal{L}_{j}^{k}\left( \mathbf{1}_{Q_{j}^{k}\setminus
\Omega _{k+m+m_{0}}}f\sigma \right) d\omega .  \label{proceed}
\end{equation}%
Here the positive linear operator $\mathcal{L}_{j}^{k}$ given by
\begin{equation*}
\mathcal{L}_{j}^{k}\left( h\sigma \right) \left( x\right) :=\sum_{\ell
=1}^{\infty }\frac{1}{|I_{j}^{k}(\ell )|}\int_{I_{j}^{k}(\ell )}hd\sigma
\mathbf{1}_{I_{j}^{k}(\ell )}(x),
\end{equation*}%
where $I_{j}^{k}\left( \ell \right) \in \mathcal{D}^{\gamma }(Q_{j}^{k})$
are the maximal dyadic cubes contained in $\mathcal{H}_{j,\mathrm{out}}^{k}$%
, which implies that $\mathcal{L}_{j}^{k}(\mathbf{1}_{Q_{j}^{k}\setminus
\Omega _{k+m+m_{0}}}f\sigma )\eqsim 2^{k}\mathbf{1}_{\mathcal{H}_{j,\mathrm{%
out}}^{k}}$. Indeed, as we have calculated,
\begin{equation}
\frac{1}{|Q_{j}^{k}|}\int_{Q_{j}^{k}}f\sigma \leq 2^{k+m-2},  \label{already}
\end{equation}%
so in particular, $I_{j}^{k}(\ell )$'s are proper subcubes of $Q_{j}^{k}$
and the claim follows. Now we can continue from (\ref{proceed}) as follows:%
\begin{align*}
\sum_{(k,j)\in \Pi _{2}}2^{k}\int_{\widetilde{E}_{j,\gamma }^{k}}\mathcal{L}%
_{j}^{k}\big(\mathbf{1}_{Q_{j}^{k}\setminus \Omega _{k+m+m_{0}}}& f\sigma %
\big)d\omega =\sum_{(k,j)\in \Pi _{2}}2^{k}\int_{Q_{j}^{k}\setminus \Omega
_{k+m+m_{0}}}\mathcal{L}_{j}^{k}\big(\mathbf{1}_{\widetilde{E}_{j,\gamma
}^{k}}\omega \big)fd\sigma \\
& \leq \sum_{(k,j)\in \Pi _{2}}2^{k}\Big(\int_{Q_{j}^{k}\setminus \Omega
_{k+m+m_{0}}}\mathcal{L}_{j}^{k}\big(\mathbf{1}_{\widetilde{E}_{j,\gamma
}^{k}}\omega \big)^{2}d\sigma \Big)^{\frac{1}{2}}\Big(\int_{Q_{j}^{k}%
\setminus \Omega _{k+m+m_{0}}}f^{2}d\sigma \Big)^{\frac{1}{2}} \\
& \leq \Big(\sum_{(k,j)\in \Pi _{2}}2^{2k}\int_{Q_{j}^{k}\setminus \Omega
_{k+m+m_{0}}}\mathcal{L}_{j}^{k}\big(\mathbf{1}_{\widetilde{E}_{j,\gamma
}^{k}}\omega \big)^{2}d\sigma \Big)^{\frac{1}{2}}\Big(\sum_{(k,j)\in \Pi
_{2}}\int_{Q_{j}^{k}\setminus \Omega _{k+m+m_{0}}}f^{2}d\sigma \Big)^{\frac{1%
}{2}} \\
& \leq \Big(\sum_{(k,j)\in \Pi _{2}}2^{2k}\int_{Q_{j}^{k}}\mathcal{L}_{j}^{k}%
\big(\mathbf{1}_{Q_{j}^{k}}\omega \big)^{2}d\sigma \Big)^{\frac{1}{2}}\Big(%
\sum_{k\in \mathbb{Z}}\int_{\Omega _{k}\setminus \Omega
_{k+m+m_{0}}}2f^{2}d\sigma \Big)^{\frac{1}{2}} \\
& \leq C_{m,m_{0}}A_{2}^{\frac{1}{2}}\Big(\sum_{(k,j)\in \Pi
_{2}}2^{2k}\left\vert Q_{j}^{k}\right\vert _{\omega }\Big)^{\frac{1}{2}%
}\Vert f\Vert _{L^{2}(\sigma )} \\
& \leq \beta ^{-\frac{1}{2}}C_{m,m_{0}}A_{2}^{\frac{1}{2}}\Big(%
\sum_{(k,j)\in \Pi _{2}}2^{2k}|\widetilde{E}_{j,\gamma }^{k}|_{\omega }\Big)%
^{\frac{1}{2}}\Vert f\Vert _{L^{2}(\sigma )},
\end{align*}%
where we have used the following trivial estimate
\begin{equation}
\int_{Q_{j}^{k}}\mathcal{L}_{j}^{k}\big(\mathbf{1}_{Q_{j}^{k}}\omega \big)%
^{2}d\sigma \leq \sum_{\ell =1}^{\infty }\frac{|I_{j}^{k}(\ell )|_{\omega
}|I_{j}^{k}(\ell )|_{\sigma }}{|I_{j}^{k}(\ell )|^{2}}|I_{j}^{k}(\ell )\cap
Q_{j}^{k}|_{\omega }\leq A_{2}|Q_{j}^{k}|_{\omega }.  \label{eq:a2ljk}
\end{equation}%
Then immediately we get
\begin{equation}
\sum_{(k,j)\in \Pi _{2}}2^{2k}|\widetilde{E}_{j,\gamma }^{k}|_{\omega }\leq
\beta ^{-1}C_{m+m_{0}}^{2}A_{2}\Vert f\Vert _{L^{2}(\sigma )}^{2}.
\label{case 2 est}
\end{equation}

\textbf{Case (3)}: For this case, we let $\{I_{j}^{k}(\ell )\}_{\ell }$ be
the collection of the maximal dyadic cubes in $\mathcal{H}_{j,\mathrm{in}%
}^{k}$ and define $\mathcal{L}_{j}^{k}$ similarly. Then likewise, $\mathcal{L%
}_{j}^{k}(\mathbf{1}_{Q_{j}^{k}\cap \Omega _{k+m+m_{0}}}f\sigma )\eqsim 2^{k}%
\mathbf{1}_{\mathcal{H}_{j,\mathrm{in}}^{k}}$ and therefore,
\begin{eqnarray*}
\sum_{(k,j)\in \Pi _{3}}2^{2k} \vert \widetilde E_{j,\gamma}^{k} \vert
_{\omega } &\lesssim &\sum_{(k,j)\in \Pi _{3}}2^{k}\int_{\widetilde
E_{j,\gamma}^{k}}\mathcal{L}_{j}^{k}\left( \mathbf{1}_{Q_{j}^{k}\cap \Omega
_{k+m+m_{0}}}f\sigma \right) d\omega \\
&=&\sum_{(k,j)\in \Pi _{3}}2^{k}\int_{Q_{j}^{k}\cap \Omega _{k+m+m_{0}}}%
\mathcal{L}_{j}^{k}\left( \mathbf{1}_{\widetilde E_{j,\gamma}^{k}}\omega
\right) fd\sigma \\
&=&\sum_{(k,j)\in \Pi _{3}}2^{k}\sum_{i\in \mathbb{N}:\
Q_{i}^{k+m+m_{0}}\subset Q_{j}^{k}}\int_{Q_{i}^{k+m+m_{0}}}\mathcal{L}%
_{j}^{k}\left( \mathbf{1}_{\widetilde E_{j,\gamma}^{k}}\omega \right)
fd\sigma .
\end{eqnarray*}

Before moving on, let us make some observations. Since we only need to
consider $I_{j}^{k}(\ell )$ such that $I_{j}^{k}(\ell )\cap \widetilde
E_{j,\gamma}^{k}\neq \emptyset $, we have $I_{j}^{k}(\ell )\not\subset
\Omega _{k+m+m_{0}}$. Therefore, if we fix $Q_{i}^{k+m+m_{0}}$, only those $%
I_{j}^{k}(\ell )$ such that $Q_{i}^{k+m+m_{0}}\subset I_{j}^{k}(\ell )$
contribute to $\mathcal{L}_{j}^{k}$. In other words, $\mathcal{L}_{j}^{k}%
\big( \mathbf{1}_{\widetilde E_{j,\gamma}^{k}}\omega \big) $ is constant on $%
Q_{i}^{k+m+m_{0}}$. Set
\begin{equation}
A_{j}^{k}:=\frac{1}{\left\vert Q_{j}^{k}\right\vert _{\sigma }}%
\int_{Q_{j}^{k}}fd\sigma .  \label{not con}
\end{equation}%
We have
\begin{align*}
\sum_{(k,j)\in \Pi _{3}}2^{2k}\left\vert \widetilde
E_{j,\gamma}^{k}\right\vert _{\omega } & \lesssim \sum_{(k,j)\in \Pi
_{3}}2^{k}\sum_{i\in \mathbb{N}:\ Q_{i}^{k+m+m_{0}}\subset
Q_{j}^{k}}A_{i}^{k+m+m_{0}}\int_{Q_{i}^{k+m+m_{0}}}\mathcal{L}_{j}^{k}\big(
\mathbf{1}_{\widetilde E_{j,\gamma}^{k}}\omega \big) \sigma \\
& =\lim_{N_0\rightarrow -\infty }\sum_{\substack{ k\in \mathbb{Z},k\geq N_0
\\ j\in \mathbb{N},(k,j)\in \Pi _{3}}}2^{k}\sum_{i\in \mathbb{N}:\
Q_{i}^{k+m+m_{0}}\subset Q_{j}^{k}}A_{i}^{k+m+m_{0}}\int_{Q_{i}^{k+m+m_{0}}}%
\mathcal{L}_{j}^{k}\big( \mathbf{1}_{\widetilde E_{j,\gamma}^{k}}\omega %
\big) \sigma .
\end{align*}%
We make a convention that the summation over $k$ is understood as $k\equiv
k_{0}\,\mathrm{mod}\,(m+m_{0})$ for some fixed $0\leq k_{0}\leq m+m_{0}-1$,
and since we are summing over products with factor $|\widetilde
E_{j,\gamma}^{k}|_{\omega }$, without loss of generality we only consider $%
Q_{j}^{k}$ for the largest $k$ if it is repeated, and define
\begin{equation*}
\mathscr{W}^{\gamma }:=\{Q_{j}^{k}\in \mathcal{D}^\gamma:k\equiv k_{0}\,%
\mathrm{mod}\,\left( m+m_{0}\right) ,k\geq N_0, (k,j)\in \Pi_3\}.
\end{equation*}%
So in particular, there are no repeated cubes in $\mathscr{W}^{\gamma }$.
%and $\mathcal{W}^{\gamma }$ is in one-to-one correspondence with the set $W_{%
%\limfunc{dis}}^{\gamma }$ of distinguished pairs $\left( k,j\right) $ where $%
%k$ is the largest $k$ among repeated cubes, and $k\equiv k_{0}\,\mathrm{mod}%
%\,\left( m+m_{0}\right) $.
%\begin{notation}
%We say that the cube $Q_{j}^{k}$ belongs to a set $\Lambda \subset W_{%
%\limfunc{dis}}^{\gamma }$ of distinguished pairs of indices when we have $%
%\left( k,j\right) \in \Lambda $, i.e. we do not distinguish between the
%distinguished index $\left( k,j\right) $ and the corresponding cube $%
%Q_{j}^{k}$ for $\left( k,j\right) \in W_{\limfunc{dis}}^{\gamma }$. Thus if
%we write $Q\in \Lambda $, this means that $Q=Q_{j}^{k}$ for $\left(
%k,j\right) \in \Lambda $, and conversely we write $Q_{j}^{k}\in \Lambda $ if
%$\left( k,j\right) \in \Lambda $.
%\end{notation}
We have, using $|\widetilde E_{j,\gamma}^{k}|_{\omega }\approx
|3Q_{j}^{k}|_{\omega }$ and $\mathcal{L}_{j}^{k}(\mathbf{1}_{Q_{j}^{k}\cap
\Omega _{k+m+m_{0}}}f\sigma )\eqsim 2^{k}\mathbf{1}_{\mathcal{H}_{j,\mathrm{%
in}}^{k}}$ for $\left( k,j\right) \in \Pi _{3}$ again, it suffices to prove
that%
\begin{align}
& \sum_{Q_j^k\in \mathscr W^\gamma}2^{k}\sum_{i\in \mathbb{N}:\
Q_{i}^{k+m+m_{0}}\subset Q_{j}^{k}}A_{i}^{k+m+m_{0}}\int_{Q_{i}^{k+m+m_{0}}}%
\mathcal{L}_{j}^{k}\left( \mathbf{1}_{\widetilde E_{j,\gamma}^{k}}\omega
\right) \sigma  \label{last sum} \\
& \lesssim \sum_{Q_j^k\in \mathscr W^\gamma}\frac{|\widetilde
E_{j,\gamma}^{k}|_{\omega }}{|3Q_{j}^{k}|_{\omega }^{2}}\Big[ \sum_{i\in
\mathbb{N}:\ Q_{i}^{k+m+m_{0}}\subset Q_{j}^{k}}\hspace{-0.4cm}%
A_{i}^{k+m+m_{0}}\int_{Q_{i}^{k+m+m_{0}}}\mathcal{L}_{j}^{k}\left( \mathbf{1}%
_{\widetilde E_{j,\gamma}^{k}}\omega \right) \sigma \Big] ^{2}\lesssim (%
\mathfrak{T} _{M}\left( 3\right) \left( \sigma ,\omega \right) +\sqrt{
A_{2}\left( \sigma ,\omega \right) })^2\|f\|_{L^2(\sigma)}^2.  \notag
\end{align}%
For notational convenience, set
\begin{eqnarray*}
III^{\ast } &:= &\sum_{Q_j^k\in \mathscr W^\gamma}III^{\ast }\left(
Q_{j}^{k}\right) ; \\
III^{\ast }\left( Q_{j}^{k}\right) &:= &\frac{|\widetilde
E_{j,\gamma}^{k}|_{\omega }}{|3Q_{j}^{k}|_{\omega }^{2}}\Big[ \sum_{i\in
\mathbb{N}:\ Q_{i}^{k+m+m_{0}}\subset Q_{j}^{k}}\hspace{-0.4cm}%
A_{i}^{k+m+m_{0}}\int_{Q_{i}^{k+m+m_{0}}}\mathcal{L}_{j}^{k}\left( \mathbf{1}%
_{\widetilde E_{j,\gamma}^{k}}\omega \right) \sigma \Big] ^{2}.
\end{eqnarray*}

\subsection{Principal cube decomposition}

The subsection is to define principal cubes. Although we have reduced the
summation over good cubes, we still define the stopping collection which
admits bad cubes. To be precise, denote
\begin{equation*}
\mathcal{W}^{\gamma }:=\{Q_{j}^{k}\in \mathcal{W}_{k}^{\gamma }:k\equiv
k_{0}\,\mathrm{mod}\,\left( m+m_{0}\right) ,k\geq N_{0}\}.
\end{equation*}%
With the grid $\mathcal{W}=\mathcal{W}^{\gamma }$ in hand, we now introduce
principal cubes as in \cite[page 804]{MuWh} (note that we are suppressing
the dependence of $\mathcal{W}$ on $\gamma $ for reduction of notation).
Define $G_{0}$ to consist of the maximal cubes in $\mathcal{W}$. If $G_{n}$
has been defined, let $G_{n+1}$ consist \ of those indices $\left(
k,j\right) $ for which $Q_{j}^{k}\in \mathcal{W}$, there is an index $\left(
t,u\right) \in G_{n}$ with $k\geq t$ and $Q_{j}^{k}\subset Q_{u}^{t}$, and

(\textbf{i}) $A_{j}^{k}>\eta A_{u}^{t}\ $,

(\textbf{ii}) $A_{i}^{\ell }\leq \eta A_{u}^{t}$ whenever $%
Q_{j}^{k}\subsetneqq Q_{i}^{\ell }\subset Q_{u}^{t}$ .

Here $\eta $ is any constant larger than $1$, for example $\eta =4$ works
fine. Now define $\Gamma \equiv \bigcup\limits_{n=0}^{\infty }G_{n}$ and for
each index $\left( k,j\right) $ define $P\left( Q_{j}^{k}\right) $ to be the
smallest dyadic cube $Q_{u}^{t}$ containing $Q_{j}^{k}$ and with $\left(
t,u\right) \in \Gamma $. Then we have%
\begin{eqnarray}
&&\ \text{(\textbf{i}) }P\left( Q_{j}^{k}\right) =Q_{u}^{t}\Longrightarrow
A_{j}^{k}\leq \eta A_{u}^{t}\ ,  \label{main have} \\
&&\ \text{(\textbf{ii}) }Q_{j}^{k}\subsetneqq Q_{u}^{t}\text{ with }\left(
k,j\right) ,\left( t,u\right) \in \Gamma \Longrightarrow A_{j}^{k}>\eta
A_{u}^{t}\ .  \notag
\end{eqnarray}

Now we return to the estimate of $III^{\ast }$. Splitting the sum over $i$
inside $III^{\ast }\left( Q_{j}^{k}\right) $ according to whether $%
(k+m+m_{0},i)\in \Gamma $ or $P( Q_{i}^{k+m+m_{0}})=P\left( Q_{j}^{k}\right)
$:
\begin{align*}
III^{\ast } & \lesssim \sum_{Q_j^k\in \mathscr W^\gamma}\frac{|\widetilde
E_{j,\gamma}^{k}|_{\omega }}{|3Q_{j}^{k}|_{\omega }^{2}}\Big[ \sum_{i\in
\mathbb{N}:\ P\left( Q_{i}^{k+m+m_{0}}\right) =P\left( Q_{j}^{k}\right) }%
\hspace{-0.4cm}A_{i}^{k+m+m_{0}}\int_{Q_{i}^{k+m+m_{0}}}\mathcal{L}%
_{j}^{k}\left( \mathbf{1}_{\widetilde E_{j,\gamma}^{k}}\omega \right) \sigma %
\Big] ^{2} \\
& +\sum_{Q_j^k\in \mathscr W^\gamma}\frac{|\widetilde
E_{j,\gamma}^{k}|_{\omega }}{|3Q_{j}^{k}|_{\omega }^{2}}\Big[ \sum
_{\substack{ i\in \mathbb{N}:\ (k+m+m_{0},i)\in \Gamma  \\ %
Q_{i}^{k+m+m_{0}}\subset Q_{j}^{k}}}A_{i}^{k+m+m_{0}}\int_{Q_{i}^{k+m+m_{0}}}%
\mathcal{L}_{j}^{k}\left( \mathbf{1}_{\widetilde E_{j,\gamma}^{k}}\omega
\right) \sigma \Big] ^{2} \\
& =:IV+V.
\end{align*}%
It is relatively easy to estimate term $V$ by the Cauchy-Schwarz inequality
and \eqref{eq:a2ljk},
\begin{align}
V& =\sum_{Q_j^k\in \mathscr W^\gamma}\frac{|\widetilde
E_{j,\gamma}^{k}|_{\omega }}{|3Q_{j}^{k}|_{\omega }^{2}}\Big[\sum_{\substack{
i\in \mathbb{N}:\ (k+m+m_{0},i)\in \Gamma  \\ Q_{i}^{k+m+m_{0}}\subset
Q_{j}^{k}}}A_{i}^{k+m+m_{0}}\int_{Q_{i}^{k+m+m_{0}}}\mathcal{L}_{j}^{k}\big(%
\mathbf{1}_{\widetilde E_{j,\gamma}^{k}}\omega \big)\sigma \Big]^{2}
\label{V est} \\
& \leq \sum_{Q_j^k\in \mathscr W^\gamma}\frac{ \vert \widetilde
E_{j,\gamma}^{k}\vert _{\omega }}{\left\vert 3Q_{j}^{k}\right\vert _{\omega
}^{2}}\Big[\sum_{\substack{ i\in \mathbb{N}:\ (k+m+m_{0},i)\in \Gamma  \\ %
Q_{i}^{k+m+m_{0}}\subset Q_{j}^{k}}}|Q_{i}^{k+m+m_{0}}|_{\sigma }\big(%
A_{i}^{k+m+m_{0}}\big)^{2}\Big]  \notag \\
& \times \Big[\sum_{\substack{ i\in \mathbb{N}:\ (k+m+m_{0},i)\in \Gamma  \\ %
Q_{i}^{k+m+m_{0}}\subset Q_{j}^{k}}}\Big(\int_{Q_{i}^{k+m+m_{0}}}\mathcal{L}%
_{j}^{k}\big(\mathbf{1}_{\widetilde E_{j,\gamma}^{k}}\omega \big)d\sigma %
\Big)^{2}|Q_{i}^{k+m+m_{0}}|_{\sigma }^{-1}\Big]  \notag \\
& \leq \sum_{Q_j^k\in \mathscr W^\gamma}\frac{ \vert \widetilde
E_{j,\gamma}^{k} \vert _{\omega }}{\left\vert 3Q_{j}^{k}\right\vert _{\omega
}^{2}}\Big[\sum_{\substack{ i\in \mathbb{N}:\ (k+m+m_{0},i)\in \Gamma  \\ %
Q_{i}^{k+m+m_{0}}\subset Q_{j}^{k}}}|Q_{i}^{k+m+m_{0}}|_{\sigma }\big(%
A_{i}^{k+m+m_{0}}\big)^{2}\Big]\int_{Q_{j}^{k}}\big[\mathcal{L}_{j}^{k}\big(%
\mathbf{1}_{Q_{j}^{k}}\omega \big)\big]^{2}d\sigma  \notag \\
& \lesssim A_{2}\sum_{(t,u)\in \Gamma }(A_{u}^{t})^{2}|Q_{u}^{t}|_{\sigma
}\lesssim A_{2}\Vert f\Vert _{L^{2}(\sigma )}^{2}.  \notag
\end{align}

Thus we are left to estimate term $IV$. Fix $(t,u)$, and consider the sum%
\begin{eqnarray*}
IV^{t,u} &:= &\sum_{Q_{j}^{k}\in \mathscr{W}^\gamma: P(Q_j^k)=Q_u^t}\frac{%
|\widetilde E_{j,\gamma}^{k}|_{\omega }}{|3Q_{j}^{k}|_{\omega }^{2}}\Big[ %
\sum_{i\in \mathbb{N}:\ P(Q_{i}^{k+m+m_{0}})=P(Q_{j}^{k})}\hspace{-0.4cm}%
A_{i}^{k+m+m_{0}}\int_{Q_{i}^{k+m+m_{0}}}\mathcal{L}_{j}^{k}\left( \mathbf{1}%
_{\widetilde E_{j,\gamma}^{k}}\omega \right) \sigma \Big] ^{2} \\
&\lesssim & \left( A_{u}^{t}\right) ^{2}\sum_{Q_{j}^{k}\in \mathscr{W}%
^\gamma: P(Q_j^k)=Q_u^t}\frac{|\widetilde E_{j,\gamma}^{k}|_{\omega }}{%
|3Q_{j}^{k}|_{\omega }^{2}}\Big[ \int_{Q_{j}^{k}}\mathcal{L}_j^k( \mathbf{1}%
_{Q_j^k}\sigma ) \omega \Big] ^{2} \\
&=:&\left( A_{u}^{t}\right) ^{2}\mathcal{S}^{t,u},
\end{eqnarray*}%
where
\begin{equation*}
\mathcal{S}^{t,u}:=\sum_{Q_{j}^{k}\in \mathscr{W}^\gamma: P(Q_j^k)=Q_u^t}%
\frac{|\widetilde E_{j,\gamma}^{k}|_{\omega }}{|3Q_{j}^{k}|_{\omega }^{2}}%
\Big[ \int_{Q_{j}^{k}}\mathcal{L}_j^k( \mathbf{1}_{Q_j^k}\sigma ) \omega %
\Big] ^{2}.
\end{equation*}%
It is here in estimating $\mathcal{S}^{t,u}$, that the only quantitative use
of the triple testing condition occurs.

\begin{lemma}
\label{Ki}We claim that%
\begin{equation}
\mathcal{S}^{t,u}\leq C\left( \left( \mathfrak{T}_{M}\left( 3\right) \right)
^{2}+A_{2}\right) \left\vert Q_{u}^{t}\right\vert _{\sigma }\ .
\label{nontrip good bound}
\end{equation}
\end{lemma}

\begin{proof}
Let $\left\{ K_{i}\right\} _{i\in \mathcal{I}}$ be the collection of maximal
$\mathcal{D}^{\gamma }$-children $K_{i}$ satisfying $5K_{i}\subset Q_{u}^{t}$%
. If $\ell (Q_{j}^{k})<2^{-r}\ell (Q_{u}^{t})$, since $Q_{j}^{k}$ is good,
we have $\mathrm{dist}(Q_{j}^{k},\partial Q_{u}^{t})>4\ell (Q_{j}^{k})$.
Hence $5Q_{j}^{k}\subset Q_{u}^{t}$ and in particular $Q_{j}^{k}\subset
K_{i} $ for some $i\in \mathcal{I}$. For all cubes $K_{i}$ we have%
\begin{align*}
\sum_{Q_{j}^{k}\in \mathscr{W}^{\gamma }:Q_{j}^{k}\subset K_{i}}\frac{|%
\widetilde{E}_{j,\gamma }^{k}|_{\omega }}{|3Q_{j}^{k}|_{\omega }^{2}}\Big[%
\int_{Q_{j}^{k}}\mathcal{L}_{j}^{k}(\mathbf{1}_{Q_{j}^{k}}\sigma )\omega %
\Big]^{2}& \leq \sum_{Q_{j}^{k}\in \mathscr{W}^{\gamma },Q_{j}^{k}\subset
K_{i}}|\widetilde{E}_{j,\gamma }^{k}|_{\omega }\left[ \frac{1}{\left\vert
Q_{j}^{k}\right\vert _{\omega }}\int_{Q_{j}^{k}}\mathbf{1}_{K_{i}}M\left(
\mathbf{1}_{K_{i}}\sigma \right) d\omega \right] ^{2} \\
& \leq C\int \left[ M_{\omega }^{\mathcal{D}^{\gamma }}\big(\mathbf{1}%
_{K_{i}}M\left( \mathbf{1}_{K_{i}}\sigma \right) \big)\right] ^{2}d\omega \\
& \leq C\int_{K_{i}}M\left( \mathbf{1}_{K_{i}}\sigma \right) ^{2}d\omega
\lesssim \left( \mathfrak{T}_{M}\left( 3\right) \right) ^{2}|3K_{i}|_{\sigma
}\ .
\end{align*}%
Thus we have%
\begin{equation*}
\sum_{i\in \mathcal{I}}\sum_{Q_{j}^{k}\in \mathscr{W}^{\gamma
}:Q_{j}^{k}\subset K_{i}}\frac{|\widetilde{E}_{j,\gamma }^{k}|_{\omega }}{%
|3Q_{j}^{k}|_{\omega }^{2}}\Big[\int_{Q_{j}^{k}}\mathcal{L}_{j}^{k}(\mathbf{1%
}_{Q_{j}^{k}}\sigma )\omega \Big]^{2}\leq \sum_{i\in \mathcal{I}}\left(
\mathfrak{T}_{M}\left( 3\right) \right) ^{2}\left\vert 3K_{i}\right\vert
_{\sigma }\leq C_{\mathrm{bound}}\left( \mathfrak{T}_{M}\left( 3\right)
\right) ^{2}\left\vert Q_{u}^{t}\right\vert _{\sigma }\ ,
\end{equation*}%
where $C_{\mathrm{bound}}$ is a constant such that $\sum_{i\in \mathcal{I}}%
\mathbf{1}_{3K_{i}}\leq C_{\mathrm{bound}}\mathbf{1}_{Q_{u}^{t}}$ (due to a
similar argument as (iii) of Proposition \ref{prop:whitney}). We also have%
\begin{equation*}
\sum_{\substack{ Q_{j}^{k}\in \mathscr{W}^{\gamma }:\ Q_{j}^{k}\subset
Q_{u}^{t}  \\ \ell \left( Q_{j}^{k}\right) \geq 2^{-r}\ell \left(
Q_{u}^{t}\right) }}\frac{|\widetilde{E}_{j,\gamma }^{k}|_{\omega }}{%
|3Q_{j}^{k}|_{\omega }^{2}}\Big[\int_{Q_{j}^{k}}\mathcal{L}_{j}^{k}(\mathbf{1%
}_{Q_{j}^{k}}\sigma )\omega \Big]^{2}\leq C\sum_{\substack{ Q_{j}^{k}\in %
\mathscr{W}^{\gamma }:\ Q_{j}^{k}\subset Q_{u}^{t}  \\ \ell \left(
Q_{j}^{k}\right) \geq 2^{-r}\ell \left( Q_{u}^{t}\right) }}%
A_{2}|Q_{j}^{k}|_{\sigma }\leq CrA_{2}\left\vert Q_{u}^{t}\right\vert
_{\sigma }\ .
\end{equation*}%
Combining the above estimates we conclude the proof.
\end{proof}

Then summing over $\left( t,u\right) \in \Gamma $ we obtain%
\begin{eqnarray*}
IV &\lesssim &\left( \left( \mathfrak{T}_{M}\left( 3\right) \right)
^{2}+A_{2}\right) \sum_{\left( t,u\right) \in \Gamma }|Q_{u}^{t}|_{\sigma
}\left( A_{u}^{t}\right) ^{2} \lesssim \left( \left( \mathfrak{T}_{M}\left(
3\right) \right) ^{2}+A_{2}\right) \left\Vert f\right\Vert _{L^{2}\left(
\sigma \right) }^{2} ,  \notag
\end{eqnarray*}%
which combined with (\ref{V est}) gives%
\begin{equation}
\sum_{(k,j)\in \Pi _{3}}2^{2k}\left\vert E_{j}^{k}\right\vert _{\omega }\leq
\left( \left( \mathfrak{T}_{M}\left( 3\right) \right) ^{2}+A_{2}\right)
\left\Vert f\right\Vert _{L^{2}\left( \sigma \right) }^{2} .
\label{case 3 est}
\end{equation}

\subsubsection{Wrapup of the proof}

Now letting the integer $N_0\rightarrow -\infty $ in the construction of
principal cubes, and summing over $0\leq k_{0}\leq m+m_{0}-1$ in our
convention regarding distinguished index pairs, we obtain from (\ref{lim inf
average}) that
\begin{eqnarray}
&&\int_{\mathbb{R}^{n}}\left[ M\left( f\sigma \right) \left( x\right) \right]
^{2}d\omega \left( x\right) \lesssim \mathbb{E}_{\gamma }\int_{\mathbb{R}%
^{n}}\left[ M^{\mathcal{D}^\gamma}\left( f\sigma \right) \left( x\right) %
\right] ^{2}d\omega \left( x\right)  \label{wrap} \\
&\lesssim &\mathbb{E}_{\gamma }\Big( \sum_{\text{all }(k,j)}2^{2k}|E_{j,%
\gamma}^k|_\omega\Big) +3^{n}C_{n}2^{-2m_{0}}\int \left[ M\left( f\sigma
\right) \right] ^{2}d\omega  \notag \\
&\le & \pi_{\mathrm{good}}^{-1}\mathbb{E}_{\gamma } \Big( \sum_{(k,j)\in \Pi
_{1}}2^{2k}|\widetilde E_{j,\gamma}^k|_\omega+\sum_{(k,j)\in \Pi
_{2}}2^{2k}|\widetilde E_{j,\gamma}^k|_\omega+\sum_{(k,j)\in \Pi
_{3}}2^{2k}|\widetilde E_{j,\gamma}^k|_\omega\Big) +3^{n}C_{n}2^{-2m_{0}}%
\int \left[ M\left( f\sigma \right) \right] ^{2}d\omega ,  \notag
\end{eqnarray}%
which by the estimates (\ref{case 1 est}), (\ref{case 2 est}) and (\ref{case
3 est}) gives%
\begin{eqnarray}
&&\int_{\mathbb{R}^{n}}\left[ M\left( f\sigma \right) \left( x\right) \right]
^{2}d\omega \left( x\right)  \label{wrap'} \\
&\lesssim &\left( \beta +2^{-2m_{0}}\right) \int M\left( f\sigma \right)
^{2}d\omega +\beta ^{-1}C_{m+m_{0}}^{2}A_{2}\Vert f\Vert _{L^{2}(\sigma
)}^{2}+\left( \left( \mathfrak{T}_{M}\left( 3\right) \right)
^{2}+A_{2}\right) \Vert f\Vert _{L^{2}(\sigma )}^{2}  \notag
\end{eqnarray}%
Now we can absorb the first term on the right hand side by choosing $\beta
>0 $ sufficiently small and $m_{0}$ sufficiently large since the integral $%
\int M\left( f\sigma \right) ^{2}d\omega $ is finite. Then we take the
supremum over $f\in L^{2}\left( \sigma \right) $ with $\left\Vert
f\right\Vert _{L^{2}(\sigma )}=1$ to obtain
\begin{equation*}
\mathfrak{N}_{M}\leq C\left( \mathfrak{T}_{M}\left( 3\right) +\sqrt{A_{2}}%
\right) \ .
\end{equation*}%
As the opposite inequality is trivial, this completes the proof of Theorem %
\ref{maximal}.

\section{Weak triple testing: proof of Theorem \protect\ref{weak}\label{Sec
weak}}

This section is devote to prove Theorem \ref{weak}. We begin with a basic
observation, that is, the places where we need $\mathfrak{T}_M(3)<\infty$
are the following

\begin{enumerate}
\item qualitatively, at the beginning of the argument, in order to assume
without loss of generality that $\int M\left( f\sigma \right) ^{2}d\omega
<\infty $,

\item and quantitatively, near the end of the argument, in the proof of
Lemma \ref{Ki}.
\end{enumerate}

The qualitative use of the triple testing condition is easily handled using $%
D$-triple testing as follows. As before, we can replace $\omega $ by $\omega
_{N}=\omega \mathbf{1}_{B(0,N)}$ with $N>R$, and $f$ is supported on $%
Q\left( 0,R\right) $. Moreover, the \emph{testing condition} for the cube $%
Q_{m}=Q(0,3^{m}N)$ must hold for some $m\geq 0$, since otherwise iteration
of the inequality $\left\vert Q_{m}\right\vert _{\sigma }\leq \frac{1}{D}%
\left\vert Q_{m+1}\right\vert _{\sigma }$ eventually violates the $A_{2}$
condition,
\begin{equation*}
A_{2}\left( \sigma ,\omega \right) \geq \frac{\left\vert Q_{m}\right\vert
_{\sigma }\left\vert Q_{m}\right\vert _{\omega _{N}}}{\left\vert
Q_{m}\right\vert ^{2}}\geq \frac{D^{m}\left\vert Q_{0}\right\vert _{\sigma
}\left\vert Q_{0}\right\vert _{\omega }}{3^{2mn}\left\vert Q_{0}\right\vert
^{2}}=\left( \frac{D}{3^{2n}}\right) ^{m}\frac{\left\vert Q_{0}\right\vert
_{\sigma }\left\vert Q_{0}\right\vert _{\omega }}{\left\vert
Q_{0}\right\vert ^{2}},
\end{equation*}%
if $D$ is chosen greater than $3^{2n+1}$. Thus if the \emph{testing condition%
} holds for the cube $Q_{m}$ we have
\begin{equation*}
\int M\left( f\sigma \right) ^{2}d\omega _{N}\leq \Vert f\Vert _{L^{\infty
}}\int_{Q(0,N)}M(\mathbf{1}_{Q_{m}}\sigma )^{2}d\omega <\infty ,
\end{equation*}%
and therefore, without loss of generality, we can assume $\int M(f\sigma
)^{2}d\omega <\infty $. The second point is delicate. Notice that, if we {%
\textit{a priori}} have the usual testing, i.e., $\mathfrak{T}_{M}<\infty $,
then we trivially have
\begin{equation*}
\mathfrak{T}_{M}(3)\leq \mathfrak{T}_{M}^{D}(3)+D^{-1}\mathfrak{T}_{M}.
\end{equation*}%
Then let $D$ be sufficiently large and use the trivial fact that $\mathfrak{T%
}_{M}\leq \mathfrak{N}_{M}$ we get
\begin{equation*}
\mathfrak{N}_{M}\leq C\left( \mathfrak{T}_{M}^{D}\left( 3\right) +\sqrt{A_{2}%
}\right) .
\end{equation*}%
Therefore, the main goal is to remove the a priori assumption. To this end,
consider the truncated maximal function
\begin{equation*}
M^{t}f(x):=\sup_{Q\ni x,\ell (Q)\geq t}\frac{1}{|Q|}\int_{Q}|f|.
\end{equation*}%
Obviously, $M^{t}\leq M$ and therefore, the related testing condition for $%
M^{t}$ also holds. By the monotone convergence theorem, it suffices to prove
\begin{equation*}
\Vert M^{t}(f\sigma )\Vert _{L^{2}(\omega )}\leq C\Vert f\Vert
_{L^{2}(\sigma )}.
\end{equation*}

The argument for $M^{t}$ is basically identical with that for $M$ above.
Indeed, with the same proof, we have the following analogy of Lemma  \ref{domination}:
\[
M^tf\left( x\right) \leq 2^{4n+1}\mathbb{E}_\gamma M^{\mathcal{D}^\gamma,t}f\left(
x\right),\ \ M^{\mathcal{D}^\gamma,t}f(x):=\sup_{x\in Q\in \mathcal D^\gamma,\ell (Q)\geq t}\frac{1}{|Q|}\int_{Q}|f|.
\]
Then everything is same with $M^t$ and $M^{\mathcal{D}^\gamma,t}$ in place of $M$ and $M^{\mathcal{D}^\gamma}$, respectively.
We only remark that when $\ell (Q_j^k)<t$ we do not have the analogy of the maximal principle \eqref{max princ}, instead, the same arguments will imply that $E_j^{k,\gamma}$ (defined according to $M^t$ and $M^{\mathcal{D}^\gamma,t}$) is empty, which does not affect the proof.

Hence, for fixed $t>0$, it remains to check that the \emph{a priori}
assumption $\mathfrak{T}_{M^{t}}(\sigma ,\omega _{N})<\infty $ holds
automatically. Indeed, fix a cube $Q$ and denote $K:=Q(0,N)$, we have
\begin{align*}
\int_{Q}M^{t}(1_{Q}\sigma )^{2}d\omega _{N}& \leq 2\int_{Q}M^{t}(1_{Q\cap
3K}\sigma )^{2}d\omega _{N}+2\int_{Q}M^{t}(1_{Q\setminus 3K}\sigma
)^{2}d\omega _{N} \\
& \leq 2\int_{K}M^{t}(1_{Q\cap 3K}\sigma )^{2}d\omega +2\int_{Q\cap
K}M(1_{Q\setminus 3K}\sigma )^{2}d\omega .
\end{align*}%
To continue, notice that trivially
\begin{equation*}
M^{t}(1_{Q\cap 3K}\sigma )\leq t^{-n}\sigma (Q\cap 3K)\ \ \text{and}\ \
M(1_{Q\setminus 3K}\sigma )1_{K}\eqsim \text{constant}.
\end{equation*}%
Therefore,
\begin{align*}
\int_{Q}M^{t}(1_{Q}\sigma )^{2}d\omega _{N}& \lesssim
t^{-2n}N^{2n}A_{2}\sigma (Q)+w(K)  M (1_{Q\setminus 3K}\sigma )(c_K)^{2}
\\
& \lesssim (t^{-2n}N^{2n}+1)A_{2}\sigma (Q),
\end{align*}%
which affirms the claim and therefore the proof of Theorem \ref{weak} is
complete.

\section{Proof of Theorem \protect\ref{weakfrac}}

\label{sec: frac} This section is devoted to prove Theorem \ref{weakfrac}.
We start with the weak type inequalities for fractional integrals. Let $%
\mathfrak{N}_{I_{\alpha }}^{\mathrm{weak}}\left( \sigma ,\omega \right) $
denote the best constant in the weak type $\left( 2,2\right) $ inequality
for the fractional integral $I_{\alpha }$:%
\begin{equation*}
\sup_{\lambda >0}\lambda ^{2}\left\vert \left\{ I_{\alpha }(f\sigma)
>\lambda \right\} \right\vert _{\omega }\leq \mathfrak{N}_{I_{\alpha }}^{%
\mathrm{weak }}\left( \sigma, \omega \right) ^{2}\int \left\vert
f\right\vert ^{2}d\sigma ,\ \ \ \ \ f\in L^{2}\left( \sigma \right) .
\end{equation*}%
It is known from \cite{Saw1} that the weak type norm is equivalent to the
dual testing condition, $\mathfrak{N}_{I_{\alpha }}^{\mathrm{weak}}\left(
\sigma ,\omega \right) \approx \mathfrak{T}_{I_{\alpha }}\left( \omega
,\sigma \right) $, and also that $A_{2}^{\alpha }\left( \sigma ,\omega
\right) \leq \mathfrak{T}_{I_{\alpha }}\left( \omega ,\sigma \right) $ where
the $\alpha $-fractional Muckenhoupt condition is given by%
\begin{equation*}
A_{2}^{\alpha }\left( \sigma ,\omega \right) \equiv \sup_{Q\in \mathcal{P}%
^{n}}\frac{\left\vert Q\right\vert _{\sigma }}{\left\vert Q\right\vert ^{1-%
\frac{\alpha }{n}}}\frac{\left\vert Q\right\vert _{\omega }}{\left\vert
Q\right\vert ^{1-\frac{\alpha }{n}}}.
\end{equation*}%
We now show that the weight pair $\left( \sigma ,\omega \right) $ in Example %
\ref{no control} above also satisfies $\mathfrak{T}_{I_{\alpha }}(3)\left(
\sigma ,\omega \right) <\infty $ and $A_{2}^{\alpha }\left( \sigma ,\omega
\right) =\infty $ for $0<\alpha <1$.

\begin{example}
Define%
\begin{eqnarray*}
d\sigma \left( y\right) &=&e^{y}dy, \\
d\omega \left( x\right) &=&\mathbf{1}_{\left[ 0,1\right] }\left( x\right) dx.
\end{eqnarray*}%
Then%
\begin{equation*}
\mathfrak{N}_{I_{\alpha }}^{\mathrm{weak}}\left( \omega, \sigma \right) \geq
A_{2}^{\alpha }\left( \sigma ,\omega \right) \geq \sup_{R>1}\frac{\left\vert %
\left[ 0,R\right] \right\vert _{\omega }}{\left\vert R\right\vert ^{1-\frac{%
\alpha }{n}}}\frac{\left\vert \left[ 0,R\right] \right\vert _{\sigma }}{%
\left\vert R\right\vert ^{1-\frac{\alpha }{n}}}=\sup_{R>1}\frac{1}{R^{1-%
\frac{\alpha }{n}}}\frac{e^{R}-1}{R^{1-\frac{\alpha }{n}}}=\infty ,
\end{equation*}%
and%
\begin{equation*}
\mathfrak{T}_{I_{\alpha }}(3)\left( \sigma ,\omega \right) \lesssim 1.
\end{equation*}%
Indeed, without loss of generality, $I=\left[ a,b\right] $ with $I\cap \left[
0,1\right] \neq \emptyset $ (since otherwise $\mathbf{1}_{I}\omega =0$ and $%
\int_{I}\left\vert I_{\alpha }\left( \mathbf{1}_{I}\sigma \right)
\right\vert ^{2}d\omega =0$) and so%
\begin{equation*}
a<1\text{ and }b>0.
\end{equation*}%
Now we assume this and compute $\frac{1}{\left\vert 3I\right\vert _{\sigma }}%
\int_{I}\left\vert I_{\alpha }\left( \mathbf{1}_{I}\sigma \right)
\right\vert ^{2}d\omega $ in two cases.

\begin{enumerate}
\item Case $b>2$: In this case we have
\begin{equation*}
I_{\alpha }\left( \mathbf{1}_{I}\sigma \right) \left( x\right)
=\int_{a}^{b}\left\vert x-y\right\vert ^{\alpha -1}e^{y}dy\leq
e^x\int_{-\infty }^{b}\left\vert y\right\vert ^{\alpha -1}e^{y}dy\lesssim
b^{\alpha-1 }e^{b}
\end{equation*}
for $0\leq x\leq 1$, and so
\begin{eqnarray*}
&&\int_{I}\left\vert I_{\alpha }\left( \mathbf{1}_{I}\sigma \right)
\right\vert ^{2}d\omega \lesssim \int_{0}^{1}\left( b^{\alpha
-1}e^{b}\right) ^{2}dx\approx \frac{e^{2b}}{b^{2\left( 1-\alpha\right) }}, \\
&&\left\vert 3I\right\vert _{\sigma }=\int_{2a-b}^{2b-a}d\sigma \geq
\int_{2b-a-1}^{2b-a}e^{y}dy\approx e^{2b-a}, \\
&\Longrightarrow &\frac{1}{\left\vert 3I\right\vert _{\sigma }}%
\int_{I}\left\vert I_{\alpha }\left( \mathbf{1}_{I}\sigma \right)
\right\vert ^{2}d\omega \lesssim \frac{\frac{e^{2b}}{b^{2\left( 1-\alpha
\right) }}}{e^{2b-a}}=\frac{e^{a}}{b^{2\left( 1-\alpha \right) }}\lesssim 1.
\end{eqnarray*}

\item Case $b\leq 2$: In this case we have $I_{\alpha }\left( \mathbf{1}%
_{I}\sigma \right) \left( x\right) =\int_{a}^{b}\left\vert x-y\right\vert
^{\alpha -1}e^{y}dy\lesssim e^{2}$ for $0\leq x\leq 1$, and so we consider
two subcases.

\begin{enumerate}
\item Subcase $a\geq -1$:%
\begin{eqnarray*}
&&\int_{I}\left\vert I_{\alpha }\left( \mathbf{1}_{I}\sigma \right)
\right\vert ^{2}d\omega \lesssim e^{2}\left\vert I\cap \left[ 0,1\right]
\right\vert \\
&&\left\vert 3I\right\vert _{\sigma }=\int_{2a-b}^{2b-a}e^{y}dy\geq
e^{2a-b}3\left( b-a\right) \geq 3e^{-4}\left( b-a\right) , \\
&\Longrightarrow &\frac{1}{\left\vert 3I\right\vert _{\sigma }}%
\int_{I}\left\vert I_{\alpha }\left( \mathbf{1}_{I}\sigma \right)
\right\vert ^{2}d\omega \leq \frac{e^{2}\left\vert I\cap \left[ 0,1\right]
\right\vert }{3e^{-4}\left( b-a\right) }\leq \frac{e^{6}}{3}.
\end{eqnarray*}

\item Subcase $a<-1$: In this subcase we also have $b-a>0-\left( -1\right)
=1 $ and so%
\begin{eqnarray*}
&&\int_{I}\left\vert I_{\alpha }\left( \mathbf{1}_{I}\sigma \right)
\right\vert ^{2}d\omega \lesssim e^{2}\left\vert I\cap \left[ 0,1\right]
\right\vert \\
&&\left\vert 3I\right\vert _{\sigma }=\int_{2a-b}^{2b-a}e^{y}dy\geq
\int_{b-1}^{b}e^{y}dy=e^{b}-e^{b-1}\geq 1-e^{-1}, \\
&\Longrightarrow &\frac{1}{\left\vert 3I\right\vert _{\sigma }}%
\int_{I}\left\vert I_{\alpha }\left( \mathbf{1}_{I}\sigma \right)
\right\vert ^{2}d\omega \lesssim \frac{e^{2}\left\vert I\cap \left[ 0,1%
\right] \right\vert }{1-e^{-1}}\leq \frac{e^{2}}{1-e^{-1}}.
\end{eqnarray*}
\end{enumerate}
\end{enumerate}
\end{example}

This example shows that we cannot remove the Muckenhoupt constant $%
A_{2}^{\alpha }\left( \sigma ,\omega \right) $ from the following theorem.

\begin{theorem}
For $D>1$ sufficiently large we have%
\begin{equation*}
\mathfrak{N}_{I_{\alpha }}^{\mathrm{weak}}\left( \sigma ,\omega \right)
\approx \mathfrak{T}_{I_{\alpha }}^{D}(3)\left( \omega, \sigma \right)
+A_{2}^{\alpha }\left( \sigma ,\omega \right) ,
\end{equation*}%
for all pairs $\left( \sigma ,\omega \right) $ of locally finite positive
Borel measures on $\mathbb{R}^{n}$.
\end{theorem}

\begin{proof}
We modify the proof of the weak type characterization in \cite{Saw1}. For $f$
bounded nonnegative and having compact support, define
\begin{equation*}
\Omega _{\lambda }\equiv \left\{ I_{\alpha }(f\sigma) >\lambda \right\} = {%
\bigcup_{j}} Q_{k}
\end{equation*}
as in the standard Whitney decomposition with $N=9$. Then we have the well
known maximum principle,
\begin{equation*}
I_{\alpha }(f \mathbf{1}_{\left( 3Q_{k}\right) ^{c}}\sigma) \left( x\right)
\leq \gamma \lambda ,\ \ \ \ \ \text{for }x\in Q_{k}\cap \{
M_\alpha(f\sigma)\le \varepsilon(\gamma) \lambda\}\ .
\end{equation*}%
Denote by $E$ the set of indices $k$ such that%
\begin{equation}
\left\vert 9Q_{k}\right\vert _{\omega }>D\left\vert Q_{k}\right\vert
_{\omega }\ ,  \label{doubling}
\end{equation}%
by $F$ the set of indices $k$ such that (\ref{doubling}) fails \ and%
\begin{equation}
\frac{1}{\left\vert Q_{k}\right\vert _{\omega }}\int_{Q_{k}}I_{\alpha
}\left( \mathbf{1}_{3Q_{k}}fd\sigma \right) d\omega >\beta \lambda ,
\label{threshhold}
\end{equation}%
and by $G$ the set of indices $k$ such that (\ref{doubling}) and (\ref%
{threshhold}) fails. Then for $k$ in $F$ we have%
\begin{align*}
\lambda ^{2}\left\vert Q_{k}\right\vert _{\omega } &< \beta ^{-2}\left\vert
Q_{k}\right\vert _{\omega }^{-1}\left( \int_{Q_{k}}I_{\alpha }\left( \mathbf{%
1}_{3Q_{k}}fd\sigma \right) d\omega \right) ^{2} \\
&= \beta ^{-2}\left\vert Q_{k}\right\vert _{\omega }^{-1}\left(
\int_{3Q_{k}}I_{\alpha }\left( \mathbf{1}_{Q_{k}}d\omega \right) \ fd\sigma
\right) ^{2} \\
&\leq \beta ^{-2}\left\vert Q_{k}\right\vert _{\omega }^{-1}\left(
\int_{3Q_{k}}I_{\alpha }\left( \mathbf{1}_{Q_{k}}d\omega \right) ^{2}d\sigma
\right) \left( \int_{Q_{k}}f^{2}d\sigma \right) \\
&\leq \beta ^{-2}\left( \mathfrak{T}_{I_{\alpha }}^{D}(3)\right) ^{2}D\left(
\int_{Q_{k}}f^{2}d\sigma \right) ,
\end{align*}%
where we have used $D$-restricted testing $\int_{3Q_{k}}I_{\alpha }\left(
\mathbf{1}_{3Q_{k}}d\omega \right) ^{2}d\sigma \leq (\mathfrak{T}_{I_{\alpha
}}^{D}(3))^2D\left\vert Q_{k}\right\vert _{\omega }$ since (\ref{doubling})
fails for $k\in F$. On the other hand, for $k\in G$, we have by the maximum
principle that%
\begin{align*}
&\left\vert Q_{k}\cap \left\{ I_{\alpha }(f\sigma) >\left( \gamma +1\right)
\lambda \right\} \right\vert _{\omega } \\
&\leq \left\vert Q_{k}\cap \left\{ I_{\alpha }\left( \mathbf{1}%
_{3Q_{k}}f\sigma \right) >\lambda \right\} \right\vert _{\omega }+
\left\vert Q_{k}\cap \left\{ M_\alpha\left( f\sigma \right)
>\varepsilon(\gamma)\lambda \right\} \right\vert _{\omega } \\
&\leq \beta \left\vert Q_{k}\right\vert _{\omega }+ \left\vert Q_{k}\cap
\left\{ M_\alpha\left( f\sigma \right) >\varepsilon(\gamma)\lambda \right\}
\right\vert _{\omega }\ ,
\end{align*}%
since (\ref{doubling}) and (\ref{threshhold}) fails. Altogether this gives
the `good $\lambda $-inequality'%
\begin{align*}
&\left( 3\lambda \right) ^{2}|\big\{ I_{\alpha }(f\sigma) >3\lambda \big\} %
|_{\omega } \\
&=\sum_{k}\left( 3\lambda \right) ^{2}\left\vert Q_{k}\cap \left\{ I_{\alpha
}(f\sigma) >3\lambda \right\} \right\vert _{\omega } \\
&\le 9\sum_{k\in E}\lambda ^{2}\left\vert Q_{k}\right\vert _{\omega
}+9\sum_{k\in F}\lambda ^{2}\left\vert Q_{k}\right\vert _{\omega }+9\lambda
^{2}\sum_{k\in G}\left\vert Q_{k}\cap \left\{ I_{\alpha }(f\sigma) >3\lambda
\right\} \right\vert _{\omega } \\
&\leq \frac{9}{D}\sum_{k\in E}\lambda ^{2}\left\vert 9Q_{k}\right\vert
_{\omega }+\left( \frac{3}{\beta }\mathfrak{T}_{I_{\alpha }}^{D}(3)\right)
^{2}D\sum_{k\in F}\left( \int_{Q_{k}}f^{2}d\sigma \right) +9\lambda
^{2}\beta \sum_{k\in G}\left\vert Q_{k}\right\vert _{\omega } \\
&\qquad+ (3\lambda)^2 \left\vert \left\{ M_\alpha\left( f\sigma \right)
>\varepsilon \lambda \right\} \right\vert _{\omega } \\
&\leq \frac{9C_W}{D}\lambda ^{2}\left\vert \left\{ I_{\alpha }(f\sigma)
>\lambda \right\} \right\vert _{\omega }+\left( \frac{3}{\beta }\mathfrak{T}%
_{I_{\alpha }}^{D}(3)\right) ^{2}D\left( \int f^{2}d\sigma \right) +9\lambda
^{2}\beta \left\vert \left\{ I_{\alpha }(f\sigma) >\lambda \right\}
\right\vert _{\omega } \\
&\qquad+ C(\varepsilon) A_2^\alpha(\sigma, \omega) \int f^2 d\sigma,
\end{align*}%
where $C_W$ is a dimensional constant. If we\ now choose $\beta =\frac{1}{27}
$ and $D=27C_W$, then we obtain for each $t>0$ that%
\begin{align*}
&\sup_{t\geq \lambda >0}\lambda ^{2}\left\vert \left\{ I_{\alpha }(f\sigma)
>\lambda \right\} \right\vert _{\omega } \leq \sup_{t\geq \lambda >0}\left(
3\lambda \right) ^{2}\left\vert \left\{ I_{\alpha }f\sigma >3\lambda
\right\} \right\vert _{\omega } \\
&\leq \big(\left( 81 \mathfrak{T}_{I_{\alpha }}^{D}(3)\right) ^{2}D+
C(\varepsilon) A_2^\alpha(\sigma, \omega) \big)\left( \int f^{2}d\sigma
\right) +\frac{2}{3}\sup_{t\geq \lambda >0}\lambda ^{2}\left\vert \left\{
I_{\alpha }(f\sigma) >\lambda \right\} \right\vert _{\omega }\ ,
\end{align*}
We now claim that $\sup_{t\geq \lambda >0}\lambda ^{2}\left\vert \left\{
I_{\alpha }(f\sigma) >\lambda \right\} \right\vert _{\omega }$ is finite for
all $t>0$, which will complete the proof of the theorem after subtracting
the last term on the right hand side from both sides, and then letting $%
t\rightarrow \infty $. To prove the claim we recall that $f$ is bounded and
supported in $B(0, R)$, so that
\begin{equation*}
I_{\alpha }(f\sigma) \left( x\right)\approx \frac{1}{|x|^{n- \alpha }}
\int_{B(0, R)} fd \sigma, \qquad x \notin 3B(0, R).
\end{equation*}%
In other words,
\begin{equation*}
(3B(0, R))^c\cap\left\{ I_{\alpha }(f\sigma )>\lambda \right\}\subset B\Big(%
0, \Big(\frac {c_{n, \alpha}} \lambda \int_{B(0, R)}fd\sigma\Big)^{\frac
1{n-\alpha}}\Big)=: B(0, r_\lambda).
\end{equation*}
Then for $0< \lambda \le t$ we have%
\begin{align*}
\lambda ^{2}\left\vert \left\{ I_{\alpha }(f\sigma )>\lambda \right\}
\right\vert _{\omega } &= \lambda ^{2}\left\vert 3B(0, R)\cap\left\{
I_{\alpha }(f\sigma )>\lambda \right\} \right\vert _{\omega } + \lambda
^{2}\left\vert (3B(0, R))^c\cap\left\{ I_{\alpha }(f\sigma )>\lambda
\right\}\right\vert _{\omega } \\
&\le t^{2}\left\vert 3B(0, R) \right\vert _{\omega } + \frac{c_{n-\alpha}^2}{%
r_\lambda^{2(n-\alpha)}} \Big(\int_{B(0, R)} fd\sigma\Big)^2 |B(0,
r_\lambda)|_\omega \\
&\le t^{2}\left\vert 3B(0, R) \right\vert _{\omega } + c_{n, \alpha}^{\prime
}A_2^\alpha(\sigma, \omega) \int f^2 d~\sigma,
\end{align*}%
where in the last step we have used the fact that $r_\lambda > R$. This
proves the claim.

This completes the proof of the theorem since
\begin{equation*}
A_{2}^{\alpha }\left( \sigma ,\omega \right) +\mathfrak{T}_{I_{\alpha
}}^{D}(3)\left( \omega ,\sigma \right) \lesssim A_{2}^{\alpha }\left( \sigma
,\omega \right) +\mathfrak{T}_{I_{\alpha }}\left( \omega ,\sigma \right)
\lesssim \mathfrak{N}_{I_{\alpha }}^{\mathrm{weak}}\left( \sigma ,\omega
\right)
\end{equation*}%
is trivial.
\end{proof}

\end{document}